\documentclass[a4paper,oneside]{article}
\usepackage[utf8]{inputenc}
\usepackage{geometry}
\usepackage[UKenglish]{babel}
\usepackage{newlfont}
\usepackage{amssymb}
\usepackage{amsmath}
\usepackage{amsfonts}
\usepackage{latexsym}
\usepackage{amsthm}
\usepackage{mathrsfs}
\usepackage{stmaryrd}
\usepackage[all,cmtip]{xy}
\usepackage{caption}
\usepackage{enumerate}
\usepackage{hyperref}
\usepackage{xcolor}
\usepackage{stackrel}

\theoremstyle{plain}                    
\newtheorem{teo}{Theorem}[section]      
\newtheorem{theoremalpha}{Theorem}

\newtheorem{prop}[teo]{Proposition}
\newtheorem{cor}[teo]{Corollary}       
\newtheorem{lem}[teo]{Lemma}            
\theoremstyle{definition}               
   
\newtheorem{defin}[teo]{Definition}
\theoremstyle{remark}
\newtheorem{rmk}[teo]{Remark}
\newenvironment{dimo}
         {\textit{Proof of}}
     {\hspace*{\fill}\hspace*{\fill}\mbox{$\Box$}}

\theoremstyle{plain}                    

\newtheorem{cora}[theoremalpha]{Corollary}       
\theoremstyle{definition}               
\theoremstyle{remark}

\numberwithin{equation}{section}

\newcommand{\oo}{\mathcal{O}}
\newcommand{\mt}{\mathcal}

\newcommand{\Vc}{\mathcal V ec^d_{r,g,n}}
\newcommand{\Jc}{\mathcal J ac^d_{g,n}}
\newcommand{\Vr}{\mathcal V^d_{r,g,n}}

\newcommand{\Ex}{\mt Ext^d_{r,g,n}}
\newcommand{\Sym}{\text{Sym}^d\mathcal C_{g,n}}
\newcommand{\Ci}{\mathcal C_{g,n}}
\newcommand{\Cid}{\mathcal C^d_{g,n}}
\newcommand{\nud}{\nu^d_{r,g,n}}
\newcommand{\Mg}{\mathcal M_{g,n}}
\newcommand{\U}{V^d_{r,g,n}}
\newcommand{\Us}{U^{d}_{r,g,n}}

\newcommand{\Pic}{\text{Pic}}

\newcommand{\Br}{\text{Br}}

\begin{document}
\title{The Brauer group of the universal moduli space of vector bundles over smooth curves.}
\author{Roberto Fringuelli and Roberto Pirisi}

\date{}
\maketitle

\begin{abstract}We compute the Brauer group of the universal moduli stack of vector bundles on (possibly marked) smooth curves of genus at least three over the complex numbers. As consequence, we obtain an explicit description of the Brauer group of the smooth locus of the associated moduli space of semistable vector bundles, when the genus is at least four.
\end{abstract}
\tableofcontents 
\section*{Introduction}
The Brauer group of a variety is an important invariant. When the variety is proper and normal, the existence of non-trivial elements in the Brauer group is an obstruction to stable rationality, and it has been used to construct examples of non-rational varieties by many, including Saltman, Colliot-Th\'el\`ene, Mumford, Peyre to cite a few. In the case of moduli space of sheaves is often related to the existence of the universal family. Indeed, for the moduli space $U_C(r,L)$ of rank $r$ stable vector bundles with fixed determinant over a smooth complex curve $C$, the existence of a universal vector bundle over $U_C(r,L)\times C$ is equivalent to the Brauer group of $U_C(r,L)$ being trivial, and in \cite{BBGN}, Balaji, Biswas, Gabber and Nagaraj proved when the genus of $C$ is at least three the Brauer group of  is isomorphic to $\mathbb Z/\operatorname{gcd}(r,d)\mathbb Z$. Furthermore, it is generated by the Brauer class of the projective bundle obtained by restricting the universal projective bundle over $U_C(r,L)\times C$ to $U_{C}(r,L)\times\{ p\}$, for $p$ point in $C$.\\

This work comes from the aim to solve the same problem in the case when the curve $C$ is free to move. More precisely, there exists a complex variety $\U$ such that its points are in bijection with the so-called aut-equivalence classes of objects $(C,x_1,\ldots,x_n,E)$, where $C$ is a smooth curve of genus $g$ with $x_1,\ldots,x_n\in C$ distinct points and $E$ a vector bundle of degree $d$ and rank $r$. It is an irreducible normal quasi-projective variety. We will call $\U$ the \emph{universal moduli space of semistable vector bundles over $M_{g,n}$}. 
\begin{enumerate}[(1)]
	\item The open subset $\Us$ of smooth points in $\U$ parametrizes the isomorphism classes of regularly stable objects (see Definition \ref{regst}).
\end{enumerate}
 We explicitly compute the Brauer group of this variety.

\begin{theoremalpha}\label{teocoarse} Let $d,r,g,n\in\mathbb Z$, such that $r\geq 1$, $g\geq 4$ and $n\geq 0$. We have the following isomorphism of groups
	$$\Br\left(\Us\right)=\begin{cases}
	\mathbb Z/\operatorname{gcd}(d+r(1-g),r( d+1-g), r(2g-2))\mathbb Z, &\text{ if }n=0,\\
	\mathbb Z/\operatorname{gcd}(r,d)\mathbb Z, &\text{ if }n>0.
	\end{cases}$$
Moreover, for $k$ big enough the generator is given by the Brauer class of the projective bundle $\mathbb P(k)\to\Us$, where the fibre over a point $[(C,x_1,\ldots,x_n,E)]\in\Us$ is canonically isomorphic to $\mathbb P (H^0(C,E\otimes \omega_C^k))$. If $d\geq 2r(g-1)$, we can choose $k=0$. 
\end{theoremalpha}

The description of the generator does not appear to be \emph{a priori} compatible with the description given in \cite{BBGN}. We made this choice to have a uniform description of the generator, independent from $n$. If $n>0$, for any $1\leq i\leq n$, the generator can be given by the Brauer class of the projective bundle $\mathbb P_i\to\Us$ with fibre over $[(C,x_1,\ldots,x_n,E)]$ canonically isomorphic to $\mathbb P (E_{x_i})$, where $E_{x_i}$ is the fibre of the vector bundle $E\to C$ over the point $x_i$. This representative is compatible with the one given in \emph{loc. cit}.\\

The theorem is a consequence of a more general statement. Roughly speaking, our main result describes the Brauer groups of the moduli stacks associated to $\U$.  Before stating it, we need to introduce the objects of our study. 
\begin{enumerate}[(1)]\setcounter{enumi}{1}
	
\item The \emph{universal moduli stack} $\Vc$ of (not necessarily semistable) vector bundles of rank $r$ and degree $d$ over $\Mg$. It parametrizes the objects $(C,x_1,\ldots,x_n,E)$, where $C$ is a smooth complex curve of genus $g$ with $x_1,\ldots,x_n$ distinct points and a vector bundle $E$ of degree $d$ and rank $r$.
\item The \emph{rigidified universal moduli stack $\Vr$} of vector bundles of rank $r$ and degree $d$ over $\Mg$. It can be obtained from $\Vc$ as it follows: the group $\mathbb G_m$ naturally injects into the automorphism group of every object in $\Vc$ as multiplication by scalar on the vector bundle. These automorphisms can be removed using a procedure called $\mathbb G_m$-rigidification. The output is the new stack $\Vr$.
\item The \emph{universal $d$-th symmetric product }$\Sym$. It parametrizes the objects $(C,x_1,\ldots,x_n,D)$, where $C$ is a smooth complex curve of genus $g$ with $x_1,\ldots,x_n$ distinct points and an effective divisor $D$ of degree $d$.
\end{enumerate} 
All the stacks above are algebraic, irreducible and smooth. There exists a canonical map $\nud: \Vc\to \Vr$, which is the identity along the objects. The stack $\Vc$ together with $\nud$ becomes a $\mathbb G_m$-gerbe. In particular, the map $\nud$ defines a class $[\nud]$ in $H^2(\Vr,\mathbb G_m)$.\\

Our main result is a description of the relationships between the Brauer groups of the moduli spaces (1), (2), (3) and (4).
\begin{theoremalpha}\label{teoa}Let $d,r,g,n\in\mathbb Z$, such that $r\geq 1$, $g\geq 3$ and $n\geq 0$. We have the following
\begin{enumerate}[(i)]
\item The order of the Brauer class of the gerbe $\nud:\Vc\to \Vr$ is
$$
\begin{cases}
\operatorname{gcd}(d+r(1-g),r( d+1-g), r(2g-2)), &\text{if }n=0,\\
\operatorname{gcd}(r,d), &\text{if }n>0.
\end{cases}
$$
\item The $\mathbb G_m$-gerbe $\nud:\Vc\to \Vr$ induces an exact sequence of groups
$$
\mathbb Z \cdot [\nud]\to \Br(\Vr)\to \Br(\Vc).
$$
\item There is a canonical isomorphism $\Br(\Us)\cong\Br(\Vr)$ of Brauer groups.
\item For $d$ big enough, there exists an inclusion  $\Br(\Vc)\subset\Br(\text{Sym}^dC_{g,n})$ of Brauer groups.
\item There is a canonical inclusion $\Br(\Sym)\subset\Br(\mt M_{g,n+d})$ of Brauer groups.

\end{enumerate}
\end{theoremalpha}

When $g\geq 4$, the Brauer group of $\Mg$ vanishes (see Theorem \ref{cohmg}(ii)). So, in this range, we have a complete description of the Brauer groups.

\begin{cora}\label{corb}Let $d,r,g,n\in\mathbb Z$, such that $r\geq 1$, $g\geq 4$ and $n\geq 0$. We have the following:
	\begin{enumerate}[(i)]
		\item The Brauer class $[\nud]$ generates the Brauer group of the rigidification $\Vr$ and of $\Us$. Furthermore:
		$$\Br\left(\Us\right)\cong\Br(\Vr)=\begin{cases}
		\mathbb Z/\operatorname{gcd}(d+r(1-g),r( d+1-g), r(2g-2))\mathbb Z, &\text{ if }n=0,\\
		\mathbb Z/\operatorname{gcd}(r,d)\mathbb Z, &\text{ if }n>0.
		\end{cases}$$
		\item $\Br(\Vc)=\Br(\Sym)=0$.		
	\end{enumerate}
\end{cora}

When $r=1$ and $n=0$, Corollary \ref{corb} gives a positive answer to a question posed by M. Melo and F. Viviani (see \cite[Conjecture 6.9]{MV}), at least when $g\geq 4$.\\

Theorems \ref{teocoarse} and \ref{teoa} are a collection of the major results of this work. We remark that the theorems are stated in terms of the Brauer group $\Br(-)$. On the other hand, all the results of the next sections are stated in terms of the cohomological Brauer group $\Br'(-):=H^2(-, \mathbb{G}_m)_{\text{tors}}$. In our situation, the Brauer group would be \emph{a priori} a subgroup of the cohomological one. However, since the generators for $\Br'$ come from $\Br$, we have an equality $\Br'=\Br$ for our stacks.\\

With this in mind, we have that points \ref{teoa}$(i)$ and \ref{teoa}$(ii)$ are Theorem \ref{kern}. Point \ref{teoa}$(iii)$ is Theorem \ref{stack=coarse}. Point \ref{teoa}$(iv)$ is the combination of Proposition \ref{br=br} and Corollary \ref{sym=jac}. Point \ref{teoa}$(v)$ is Theorem \ref{brsym}. 

On the other hand, the isomorphism of Theorem \ref{teocoarse} is nothing but Corollary \ref{corb}(i). The part about the generator is precisely the content of Proposition \ref{generatorproj}.\\

The paper is organized in the following way. In Section \ref{prel}, we extend some well-known facts about the Brauer group of schemes to the context of Artin stacks. In Section \ref{vect}, we introduce the three different incarnations of the universal moduli space of vector bundles on marked curves: they are the spaces $\Us$, $\Vc$, $\Vr$ announced at the points (1), (2), (3), respectively. Furthermore, we will study the relationships between their Brauer groups. In Section \ref{bruniv}, we introduce the universal symmmetric product $\Sym$ of point (4) and complete the proof of Theorem \ref{teoa}. And in Section \ref{someresults}, we collect some auxiliary results about the moduli stack of marked curves $\Mg$ and the relative product $\Cid$ of the universal curve over $\Mg$.\\\\
\textbf{Notation.} All the schemes and stacks will be supposed to be over the complex numbers. With commutative, resp. cartesian, diagrams of stacks, we will intend in the $2$-categorical sense. The sheaves and their cohomologies will be taken with respect to the Lisse-\'etale site, unless otherwise stated.\\\\
\textbf{Acknowledgements.} The question was raised by Indranil Biswas to the first named author. We would like to thank him for suggesting the problem and for many useful observations about it. We would like to thank Ben Williams for helping us with a previous version of our main proof (see remark \ref{topol}), Angelo Vistoli and David Rydh for various helpful comments, and Eduardo Esteves, Johan Martens and Filippo Viviani for interesting discussions related to this work. The first named author was supported by EPSRC grant EP/N029828/1.

\begin{section}{Preliminaries about the Brauer group of an Artin stack.}\label{prel}

In this section we will introduce the different incarnations of the Brauer group, and establish some basic result on the (cohomological) Brauer group of an Artin stack.
	
	\begin{subsection}{Brauer group, Cohomological Brauer group, Bigger Brauer group}
		
		In the classical setting, given a Noetherian scheme $X$, the \emph{Brauer Group} $\text{Br}(X)$ is the group of Azumaya algebras over $X$, i.e. sheaves of algebras that are \'etale locally isomorphic to the endomorphism group of a vector bundle, modulo the relation that $\mt E \sim \mt E'$ if there exist vector bundles $V, V'$ on $X$ such that $\mt E \otimes \operatorname{End}(V) \simeq \mt E' \otimes \operatorname{End}(V')$.
		
		Azumaya algebras of rank $n^2$ are classified by $\mathrm{PGL_n}$-torsors, and the exact sequence 
		$$ 1 \rightarrow \mathbb{G}_m \rightarrow \mathrm{GL}_n \rightarrow \mathrm{PGL}_n \rightarrow 1$$
		induces a map $H^1(X,\mathrm{PGL}_n) \rightarrow H^2(X, \mathbb{G}_m)$ sending an azumaya algebra $\mt E$ to an element $\alpha_{\mt E} \in H^2(X, \mathbb{G}_m)$, which is always of $n$-torsion. This map is injective \cite[IV, thm 2.5]{Mi}, su that we have an inclusion $\text{Br}(X) \subseteq H^2(X, \mathbb{G}_m)_{\text{tors}}$.
		
		In the setting of schemes, it is known that under some very general hypotheses the inclusion is an isomorphism. Gabber \cite{Gir68} proved it for affine schemes, and had an unpublished proof for schemes carrying an ample line bundle. De Jong gave a new proof of the latter statement \cite{DJ} using twisted sheaves, which are some special sheaves on $\mathbb{G}_m$-gerbes.
		
		The definition of the Brauer group can be extended to an extremely general setting. In fact, given a locally ringed topos $(\mt X, \mathcal O)$ (e.g. the category $\mt X_{\text{\'et}}$ of lisse-\'etale sheaves on an Artin stack together with the structure sheaf $\mathcal{O}_{\mt X}$) one can define the Brauer group $\Br(\mt X, \mathcal O)$ and it is in general true that it injects into $H^2(\mt X, \mathbb{G}_m)$. Moreover, if $\mt X$ is connected of quasi-compact (i.e. if the Artin stack $\mt X$ has one of these properties) then the image of $\Br(\mt X, \mathcal O)$ is torsion and the inclusion $\Br(\mt X, \mathcal O) \subseteq H^2(\mt X, \mathbb{G}_m)_{\text{tors}}$ holds \cite{Gr66a}. A detailed explanation of this can be found in the extended version of Antieau and Williams' paper \cite{AW}, which can be found on the second author's website.\\
		
		In \cite{Gr66b} Grothendieck introduced the \emph{cohomological Brauer group} $\Br'(\mt X)$, which is just the torsion part of $H^2(\mt X, \mathbb{G}_m)$. As shown by the results above, under various hypotheses the cohomological Brauer group and Brauer group are isomorphic, and in general for an Artin stack $\mt X$ we have $\Br(\mt X) \subseteq \Br'(\mt X)$ unless the stack $\mt X$ is badly disconnected. Computing the cohomological Brauer group is vastly easier than computing the Brauer group in general, as the former allows for much stronger tools, such as spectral sequences, and in the case regular schemes the whole machinery of unramified cohomology.

Our computations in the following will always be on the cohomological Brauer group $\Br'$. We remark in each of the cases we consider our computation computes the ordinary Brauer group as well as the generator we give for $\Br'(\mt X)$ comes from $\Br( \mt X)$.	

We should also mention that one can construct a larger group, namely the \emph{Bigger Brauer group} $\widetilde{\text{Br}}(\mt X)$, defined by Taylor \cite{Ta} and adapted to the stack-theoretical setting by Heinloth and Schroer \cite{HS}, where the main difference is that the algebras are not required to have a unit. In their article they prove that if $\mt X$ is a Noetherian Artin stack with quasi-affine diagonal, the equality $\widetilde{\text{Br}}(\mt X)=H^2(\mt X, \mathbb{G}_m)$ holds (note that the group on the right need not be torsion).
	\end{subsection}
	
	\begin{subsection}{Some invariance results}
	
	There are some well known invariance results for the cohomological Brauer group of a regular Noetherian scheme, namely that the pullback through a vector bundle or an open immersion whose complement has high codimension are isomorphisms and that in general the pullback through an open immersion is injective. These hold for Noetherian regular Deligne Mumford stacks as well, as proven in \cite[2.5]{AM}. The rest of the section is dedicated to extending these results to the context of Artin stacks.
	
	For the first two properties, our strategy will be to prove the result for $H^2(\mt X, \mu_n)$ via local arguments, then use the fact that it is additionally true for the Picard group and the Kummer sequence $$1 \rightarrow \mu_n \rightarrow \mathbb{G}_m \xrightarrow{(-)^n} \mathbb{G}_m \rightarrow 1$$ to conclude.
	
		\begin{lem}
			Let $X$ be a scheme smooth over a field of characteristic zero, and let $F$ be a torsion sheaf. Then:
			\begin{enumerate}
				\item We have $H^i(\mathbb A_X^d, \pi^*(F)) = H^i(X,F)$ for all $i \geq 0$.
				\item Given an open subscheme $U \subset X$ whose complement has codimension at least $c$, we have $H^{i}(X,F)=H^{i}(U,F)$ for any $0 \leq i \leq 2s -1$.
			\end{enumerate}
		\end{lem}
		\begin{proof}
			The first fact is a well known consequence of the smooth base change theorem. The second fact is a consequence of cohomological purity for smooth couples, plus a stratification argument.
		\end{proof}
		
		\begin{lem}\label{rescompl}
			Let $\mt X$ be an algebraic stack smooth over a field of characteristic zero. Then:
			\begin{enumerate}
				\item Given a vector bundle $\mt E \xrightarrow{\pi} \mt X$, the pullback $\Pic(\mt X) \xrightarrow{\pi^*} \Pic(\mt E)$ is an isomorphism.
				\item Given an open substack $\mt U \subset \mt X$ whose complement has codimension at least $2$, we have $\Pic(\mt X) =\Pic(\mt U)$.
			\end{enumerate}
		\end{lem}
		\begin{proof}
			This is proven in \cite[1.9]{PTT15}.
		\end{proof}
		
		\begin{prop}\label{vec}Let $\mt X$ be a smooth Artin stack and $\mt E \xrightarrow{\pi} \mt X$ be a vector bundle over it. Then 
			the morphism of cohomological Brauer groups
			$$
			\Br'(\mathcal X)\hookrightarrow \Br'(\mt E)
			$$
			is an isomorphism.
		\end{prop}
		\begin{proof}
			It is a well known fact that if $k$ is a strictly Henselian local ring its étale cohomology with torsion coefficients vanishes, so that $H^i(\mathbb A_k^d, \mu_n)=0$ for all $i > 0$. Consequently the sheaves $\mathrm{R}^i_{\pi_*}(\mu_n)$ are zero for all $i > 0$ as they are zero on the local rings in the Liss\'e-\'Etale site. The map $\pi$ induces a Leray spectral sequence $$H^{p}(\mt X, R^q_{\pi_*}(\mu_n))\Rightarrow H^{i}(\mt E, \mu_n)$$ which colapses, so we see that $H^i(\mt E, \mu_n)=H^i(\mt X, \mu_n)$ for all $i$. Moreover, we also know that $\Pic(\mt E) = \Pic(\mt X)$. This gives rise to the following commutative diagram with exact rows 
			$$
			\xymatrix{
				\Pic(\mt X) \ar[d] \ar[r]^{\cdot n} & \Pic(\mt X) \ar[d] \ar[r] & H^2(\mt X, \mu_n) \ar[r] \ar[d] & \Br'(\mt X)_n\ar[r] \ar[d] & 0\\
				\Pic(\mt E) \ar[r]^{\cdot n} & \Pic(\mt E) \ar[r] & H^2(\mt E, \mu_n) \ar[r] & \Br'(\mt E)_n \ar[r] & 0\\
			}
			$$
			
			the first three arrow are isomorphisms, so the five lemma allows us to conclude.
		\end{proof}
		
		\begin{prop}\label{opengeq2}Let $\mt X$ be a smooth Artin stack and $\mt U \subset \mt X$ be an open substack such that $\mt U^c$ has codimension two or more in $\mt X$. Then 
			the morphism of cohomological Brauer groups
			$$
			\Br'(\mathcal X)\hookrightarrow \Br'(\mt U)
			$$
			is an isomorphism.
		\end{prop}
		\begin{proof}
			Given a torsion sheaf $F$ on $\mt X$, consider the Leray spectral sequence $$H^{p}(\mt X, R^qi_*F)\Rightarrow H^{i}(\mt U, F)$$ induced by the inclusion of $\mt U$ in $\mt X$.
			By Gabber's absolute purity theorem, we know that given a smooth scheme $X$ and an open subset $U \subset X$ whose complement has codimension at least two, we have $H^{q}(X,F)=H^{q}(U,F)$ for any torsion sheaf $F$ and $0 \leq q \leq 3$.
			This shows that $R^qi_*F=0$ for $0 \leq q \leq 3$, which implies that $H^{2}(\mt U, F)=H^{2}(\mt X, F)$. 
			
			Moreover the inclusion $\mt U$ in $\mt X$ induces an isomorphism of Picard groups. Thus, using the kummer exact sequence we get the following commutative diagram with exact rows
			$$
			\xymatrix{
				\Pic(\mt X) \ar[d] \ar[r]^{\cdot n} & \Pic(\mt X) \ar[d] \ar[r] & H^2(\mt X, \mu_n) \ar[r] \ar[d] & \Br'(\mt X)_n \ar[r] \ar[d] & 0\\
				\Pic(\mt U) \ar[r]^{\cdot n} & \Pic(\mt U) \ar[r] & H^2(\mt U, \mu_n) \ar[r] & \Br'(\mt U)_n \ar[r] & 0\\
			}
			$$
			
			and the fact that the first three vertical arrows are isomorphisms allows us to conclude as above.
		\end{proof}
		
%

The last statement can be extended directly from Deligne-Mumford stacks to Artin stacks.		
		
		\begin{prop}\label{openrestr}Let $\mathcal X$ be a regular Noetherian Artin stack, and let $\mathcal U \subset \mathcal X$ be an open substack. Then the morphism of cohomological Brauer groups
			$$
			\Br'(\mathcal X)\hookrightarrow \Br'(\mt U)
			$$
			is injective.
		\end{prop}
		\begin{proof}
			In \cite[2.5, iv]{AM} the property above is proven for Noetherian regular Deligne-Mumford stacks. The same proof holds without change in the broader context of Artin stacks.
		\end{proof}
		
	\end{subsection}
\end{section}

\begin{section}{The moduli stacks $\Vc, \Vr$ and the moduli space $\U$.}\label{vect}

In this section we will recall some facts about the \emph{universal moduli stack of vector bundles on smooth curves} $\Vc$, its rigidified version $\Vr$ and its good moduli space $\U$, then we will establish the relation between the respective cohomological Brauer groups. Through the section, we will always assume $d,r,g,n$ are integers such that $r\geq 1$, $g\geq 3$ and $n\geq 0$.

\begin{subsection}{The universal moduli stack $\Vc$.}

Recall that the moduli stack of smooth genus $g$ curves $\Mg$ is defined as the category fibered in groupoids over $(Sch/\mathbb C)$, which associates to any scheme $S$ the groupoid of objects $(C\to S,\sigma_1,\ldots,\sigma_n)$, where $C\to S$ is a family of smooth curves of genus $g$ and $\sigma_1,\ldots,\sigma_n$ are disjoint sections of the family.\\
It is known that this category is an irreducible noetherian smooth Deligne-Mumford stack of dimension $3g-3+n$. Moreover, if $n>2g+2$ it is a quasi-projective variety.

It admits a \emph{universal curve} $\pi: \Ci\rightarrow \Mg$, i.e. a stack $\Ci$ and a representable morphism $\pi$ with the property that for any morphism from a scheme $S$ to $\Mg$ associated to a pair $(C\rightarrow S,\sigma_1,\ldots,\sigma_n)$, there exists a morphism $C\rightarrow \Ci$ such that the diagram
$$
\xymatrix{
	C\ar[r]\ar[d] &\Ci\ar[d]^{\pi}\\
	S\ar[r] & \Mg}
$$
is cartesian. Moreover, the universal curve admits sections $s_i$ for $i=0,\ldots,n$, such that the pull-back of them, along the cartesian diagram above, coincides with the sections $\sigma_i$ of the family $C\to S$. 

\begin{defin}We denote by $\Vc$ the category fibered over $(Sch/k)$, which to any scheme $S$ associates the groupoid of objects $(C\to S, \sigma_1,\ldots,\sigma_n , E)$, where $(C\to S, \sigma_1,\ldots,\sigma_n )$ is a family of smooth $n$-marked curves of genus $g$ over $S$ and $E$ is an $S$-flat vector bundle over $C$ of degree $d$ and rank $r$. 
	
A morphism between two objects $(C\to S, \sigma_1,\ldots,\sigma_n, E)$, $(C'\to S, \sigma'_1,\ldots,\sigma'_n, E')$ is a pair $(\varphi,\psi)$, where $\varphi:C\cong C'$ is an isomorphism of curves preserving the markings, i.e. $\varphi(\sigma_i)=\sigma'_i$ for any $i=1,\ldots,n$, and $\psi:\varphi^*E'\cong E$ is an isomorphism of vector bundles over $C$.

There is a natural projection $\Vc \to \Mg$ given by forgetting the vector bundle.
\end{defin}

We have the following:

\begin{teo}\label{teo1}$\Vc$ is an irreducible smooth Artin stack of dimension $(r^2+3)(g-1)+n$. Furthermore, the locus of (semi)stable vector bundles is a quotient stack of a smooth quasi-projective variety by a suitable general linear group. Its complement in $\Vc$ has codimension at least two.
\end{teo}

\begin{proof}We set $\Vc:=\mt Vec_n$. We will give a proof by induction on $n$. When $n=0$, the first part is a consequence of \cite[Theorem 1.2.2]{Fr16}. Furthermore, in \emph{loc. cit.}, it has been shown that the stack $\mt Vec_0$ has an open covering $$\stackrel[k \in \mathbb{Z}]{}{\cup} \mt U_k \rightarrow \mt Vec_0,$$where $\mt U_k \subset \mt Vec_0$ is the subcategory of elements $(p:C\to S, E)$ such that
\begin{enumerate}[(i)]
\item $E(k) := E \otimes \omega_C^k$ is relatively generated by global sections, i.e. the canonical morphism $p^*p_*E(k)\to E(k)$ is surjective, and the induced morphism to the Grassmannian $C\to \text{Gr}(p_*E(k),r)$ is a closed embedding,
\item $R^ip_*E(k)=0$ for $i>0$.
\end{enumerate}
Each of the substacks $\mt U_k$ are quotient stacks of a smooth noetherian scheme by a suitable general linear group. By \cite[Subsection 1.1]{Sc01}, if $d+kr>2rg$, then $\mt U_k$ contains all (semi)stable objects. It is known that the (semi)stability is an open condition, so the (semi)stable objects describe an open in $\mt U_k$ for some $k$. And the open of a quotient stack is still a quotient stack, then we have the second assertion. Moreover, in the proof of \cite[Lemma 3.1.5]{Fr16}, it is shown that the locus of non (semi)stable bundles is a closed of codimension at least two in $\mt Vec_0$.

Assume now that the theorem holds for $\mt Vec_{n-1}^d$. Let $\mt C$ be the category fibered in groupoids parametrizing the objects $(C,x_1,\ldots,x_{n-1},y,E)$, where $(C,x_1,\ldots,x_{n-1},E)$ is an object of $\mt Vec_{n-1}$ and $y$ is any point in $C$. Adapting the proofs of \cite[Proposition 1.2.4]{Fr16} at the case of smooth curves with marked points, we see that $\mt C$ is a stack and its diagonal is representable, quasi-compact e separated. By definition, it has a natural map $\pi:\mt C\to \mt Vec_{n-1}$, which forgets the $n$-th point. It can be shown that such map is a representable morphism, furthermore it is smooth, proper, with geometrically integral fibres and of relative dimension $1$. This makes $\mt C$ an irreducible smooth Artin stack of dimension $(r^3+3)(g-1)+n$. Since $\mt Vec_n$ can be identified with the open subset of $\mt C$ of those objects where all the points are distinct, we have the first assertion. The other assertions follow by the inductive hypothesis and the representability of the forgetful map $\mt Vec_n\to \mt Vec_{n-1}$.
\end{proof}

\begin{rmk}\label{capo} In the case $r=1$, $\Vc$ is quasi compact and it is the so-called universal Jacobian over $\Mg$. According to the notation of \cite{MV}, we will set $\Jc:=\mathcal Vec^d_{1,g,n}$.
\end{rmk}

By pulling back the universal curve over $\Mg$, we obtain a \emph{universal curve} $\mt C\rightarrow \Vc$ (it is the auxiliary stack introduced in the proof of Theorem \ref{teo1}). It comes equipped with the \emph{universal sections} $s_1,\ldots,s_n$. Furthermore, the universal curve admits a \emph{universal vector bundle}, i.e. a coherent sheaf $\mt E$ over $\mt C$, flat over $\Vc$, such that for any morphism from a scheme $S$ to $\Vc$ associated to a pair $(C\rightarrow S,s_1,\ldots,s_n, E)$, the restriction of $\mt E$ to $C$ is isomorphic to $E$.

\end{subsection}
\begin{subsection}{The rigidified moduli stack $\Vr$.}

The group $\mathbb G_m$ is contained in a natural way in the automorphism group of any object of $\Vc$, as multiplication by scalars on the vector bundle. There exists a procedure for removing these automorphisms, called \emph{$\mathbb{G}_m$-rigidification} (see \cite[Section 5]{ACV}). We obtain an irreducible smooth Artin stack $\Vr:=\Vc\fatslash \mathbb{G}_m$ of dimension $(r^2+3)(g-1)+n+1$, with a surjective smooth morphism $\nud:\Vc\rightarrow\Vr$.

Observe that the above map endows $\Vc$ with a $\mathbb G_m$-gerbe structure over $\Vr$.  It is well-known that to any $\mathbb G_m$-gerbe $g:\mt Y\to\mt X$ over $\mt X$, we can associate a class $[g]$ in the cohomology group $H^2(\mt X,\mathbb G_m)$ (see \cite[IV, $\S$3.4-5]{Gi}). So in particular, the map $\nud$ defines a class $[\nud]$ in $H^2(\Vr,\mathbb G_m)$.\\

The rest of the subsection is devoted to prove the following:

\begin{teo}\label{kern}The kernel of the pull-back map
	$$
	\Br'(\nud):\Br'(\Vr)\to\Br'(\Vc),
	$$
induced by $\nud:\Vc\to\Vr$, is generated by the Brauer class $[\nud]$. Furthermore, its order is $\operatorname{gcd}(d+r(1-g),r( d+1-g), r(2g-2))$ if $n=0$ and $\operatorname{gcd}(r,d)$ if $n>0$.
\end{teo}

\begin{proof}We remove temporarily $r,d,g,n$ from the notation.	The Leray spectral sequence 
	\begin{equation}\label{leray}
	H^p(\mt V,R^q\nu_*\mathbb{G}_m) \Rightarrow H^{p+q}(\mt Vec,\mathbb G_m),
	\end{equation}
	induces an exact sequence in low degrees
	$$
	0\rightarrow H^1(\mt V,\nu_*\mathbb G_m)\rightarrow  H^1(\mt Vec,\mathbb G_m)\rightarrow H^0(\mt V,R^1\nu_*\mathbb G_m)\rightarrow H^2(\mt V,\nu_*\mathbb G_m)\to H^2(\mt Vec,\mathbb G_m).
	$$
	We observe that $\nu_*\mathbb{G}_m=\mathbb G _m$ and that the sheaf $R^1\nu_*\mathbb G_m$ is the constant sheaf $H^1(\mt B\mathbb G_m,\mathbb G_m)\cong \Pic(\mt B\mathbb G_m)\cong \mathbb Z$. Via standard cocycle computation we see that exact sequence becomes
	\begin{equation}\label{lss}
	0\longrightarrow \Pic(\mt V)\longrightarrow \Pic(\mt Vec)\xrightarrow{res} \mathbb{Z}\xrightarrow{obs} H^2(\mt V,\mathbb G_m)\xrightarrow{\nu^*} H^2(\mt Vec,\mathbb G_m),
	\end{equation}
	where $res$ is the restriction on the fibers (it coincides with the weight map defined in \cite[Def. 4.1]{H07}), $obs$ is the map which sends the identity to the Brauer class $[\nu]\in H^2(\mt V,\mathbb G_m)$.
Observe that the map in the statement is the restriction of $\nu^*$ to the torsion elements.

When $n=0$, the theorem follows from the computation of the image of $res$ in \cite[Corollary 6.10(i)]{MV} when $r=1$ and \cite[Corollary 3.3.1]{Fr16} when $r>1$.

Assume now that $n>0$. We are going to compute the image of $res$. First of all, we recall, following \cite[Section 4]{H07}, how this homomorphism is defined. Fix a line bundle $\mathcal L$ on $\Vc$, the automorphism group of an object $\eta$ in $\Vc(\mathbb C)$ acts on the fiber $\mathcal L_{\eta}$. The group $\mathbb G_m$ is contained in the automorphisms group of any object in $\Vc$ as multiplication by scalars on the vector bundle. So, by restriction $\mathbb G_m$ acts on the line bundle $\mathcal L$. By the representation theory of the group scheme $\mathbb G_m$ (here we see $\mathbb G_m$ as scheme over $\Vc$), we see that it acts on each fiber by a cocharacter $t\to t^k$ for $k\in\mathbb Z$. The homomorphism $res$ is defined by sending $\mathcal L$ to $k$.

We first prove that $\operatorname{gcd}(r,d)$ divides $k$. Let $\eta:=(C,x_1,\ldots,x_n,F^{\oplus\operatorname{gcd}(r,d)})$ be a complex point in $\Vc$, where $F$ is a stable (and then simple) vector bundle of rank $r/\operatorname{gcd}(r,d)$ and degree $d/\operatorname{gcd}(r,d)$. In particular, we have an exact sequence of algebraic groups
$$
1\to GL_{\operatorname{gcd}(r,d)}\to \text{Aut}(\eta)\to \text{Aut}(C)
$$
Furthermore, the subgroup of multiplication by scalars is contained on the left-hand side as the center $\mathbb G_m\cdot \text{Id}$. Observe that $GL_{\operatorname{gcd}(r,d)}$ acts on $L_{\eta}$ by the cocharacter $A\mapsto \det(A)^m$ for some integer $m$. So, the induced action of $\mathbb G_m$ is $t\to t^{\operatorname{gcd}(r,d)m}$. Thus, $k$ is a multiple of $\operatorname{gcd}(r,d)$.

To conclude the proof, it is enough to exhibit a line bundle with weight $\operatorname{gcd}(r,d)$. Consider the universal pair
$$
(\pi:\mt C\to \Vc, s_1,\ldots, s_n,\mt E).
$$
Following \cite[Subsection 2.2]{Fr16}, we denote by $d_\pi(\mt E)$, resp. $d_\pi(\mt E\otimes \sigma_1)$, the determinant of cohomology of $\mt E$, resp. $\mt E\otimes\sigma_1$, with respect to the family $\pi$. They are line bundles on $\Vc$ and by the same arguments of the proof of [\emph{loc. cit.}, Lemma 3.3.1]:
$$
res(d_\pi(\mt E))=d+r(1-g)\text{ and } res(d_\pi(\mt E\otimes\sigma_1))=d+r(2-g).
$$
By the fact the $\operatorname{gcd}(d+r(1-g),d+r(2-g))=\operatorname{gcd}(r,d)$, there exists an integral combination of these two line bundles such that its weight is exactly $\operatorname{gcd}(r,d)$. 
\end{proof}

\end{subsection}

\begin{subsection}{The good moduli space $\U$.}

The rigidified stack $\Vr$ admits a good moduli space $\U$. Our aim is to show that the cohomological Brauer groups of the rigidified stack $\Vr$ and of the smooth locus of the moduli variety of semistable vector bundles are isomorphic. Before doing this, we collect some facts about the moduli space and its smooth locus.

\begin{defin}Let $(C,x_1,\ldots,x_n)$ be an $n$-marked smooth curve and $E$, $F$ two semistable vector bundles on $C$. We say that $(C,x_1,\ldots,x_n,E)$ and $(C,x_1,\ldots,x_n,F)$ are \emph{aut-equivalent} if the Jordan-Holder factors of $E$ and $F$ differ by an automorphism of the marked curve.
\end{defin}

We remark that when $n>2g+2$, the automorphisms group $\text{Aut}(C,x_1,\ldots,x_n)$ of a marked curve is trivial. So, in this range the aut-equivalence relation coincide with the classical $S$-equivalence relation. 

\begin{teo}\label{goodmod}The open substack in $\Vr$ of semistable vector bundles admits an irreducible normal quasi-projective variety $\U$ as good moduli space. Points on this variety are in bijection with aut-equivalence classes of objects $(C,x_1,\ldots,x_n,E)$. \end{teo}

\begin{proof}When $\Mg$ is a variety (i.e. $n>2g+2$), it follows by a general result on existence of relative moduli spaces (see \cite[Theorem 4.3.7]{HL}). When $n=0$, a proof can be found on \cite{P96}. The same strategy, with minor changes, give the result for $1\leq n\leq 2g+2$. 
\end{proof}

We want to study the smooth locus of $\U$. 
\begin{defin}\label{regst}Let $E$ be a vector bundle over a marked smooth curve $(C,x_1,\ldots,x_n)$. We say that $(C,x_1,\ldots,x_n,E)$ is \emph{regularly stable} if $E$ is stable and $\text{Aut}(C,x_1,\ldots,x_n,E)=\mathbb G_m$.
\end{defin}
If $n>2g+2$, a vector bundle is stable if and only if is regularly stable.
\begin{lem}\label{cane} The subset $\mt S^d_{r,g,n}\subset\Vr$ of objects such that $\text{Aut}(E)\neq\text{Aut}(C,x_1,\ldots,x_n, E)$ is a closed substack of codimension at least two.
\end{lem}

\begin{proof}We first show that $\mt S^d_{r,g,n}$ is closed. We have an exact sequence of groups over $\Vr$:
$$
1\to\text{Aut}_{\oo_{\Vr}}(\mt E)\to \text{Aut}_{\oo_{\Vr}}(\mt C,\sigma_1,\ldots,\sigma_n, \mt E)\to \mt G\to 0,
$$
where the first, resp. second, group is the group over $\Vr$ of the rigidified isomorphisms of the universal vector bundle $\mt E$, resp. the universal object $(\mt C,\sigma_1,\ldots,\sigma_n, \mt E)$. Since the automorphism group of marked smooth curve is finite, the map $p:\mt G\to\Vr$ is unramified. By definition, the stack $\mt S^d_{r,g,n}$ is the locus in $\Vr$ where $p$ is not isomorphism, then it is closed.

For proving the fact about the codimension, it is enough to show the lemma when $n=0$. The case $r=1$ has been proved in \cite[Lemma 3.2.2]{Fr16}. The higher rank case follows from
$
\mt S^d_{r,g,0}\subseteq\det^{-1}(\mt S^d_{1,g,0}),
$
where $\det:\Vr\to \mt V^d_{1,g,n}$ is the determinant map.
\end{proof}

\begin{prop}\label{smoothloc}The smooth locus $\Us$ of $\U$ is the locus of regularly stable objects.
\end{prop}

\begin{proof}We remove the markings from the notation. Set $S:=Sing(\U)$. The variety $\U$ and the stack $\Vr$ are isomorphic along the locus of regularly stable objects. Since $\Vr$ is smooth, if $(C,E)$ is regularly stable, then $\U$ is smooth at $[(C,E)]$.
It remains to show the following:
\begin{enumerate}[(i)]
\item if $(C,E)$ is stable but non regularly stable, then $[(C,E)]\in S$,
\item if $(C,E)$ is polystable, then $[(C,E)]\in S$.
\end{enumerate}

We first prove (i). By openness of the stable locus, we can restrict to prove the assertion along the open subset in $\U$ of stable vector bundles. With abuse of notation, we denote this open by $V$, the corresponding singular locus by $S$ and  the locus of non regularly stable objects by $N$. 
So, proving (i) is equivalent to show $S=N$.\\
By the above discussions, we already know $S\subset N$. We will show the other inclusion. Let $(\mt C,\mt E)\to B$ be the versal deformation of a stable but non-regularly stable object $(C,E)$. Since the stack $\Vr$ is smooth, also $B$ is smooth. The group $\text{Aut}:=\text{Aut}(C,E)$ acts on $B$ and it can be shown that subgroup $\mathbb G_m\subset \text{Aut}$ acts trivially on $B$. Since the object is stable, the group $\text{Aut}$ is linearly reductive. By \cite[Lemma 1.4.4]{Fr16} (it is stated for $n=0$, but it holds with the same proof even if we consider the markings), the quotient $B/\text{Aut}$ is isomorphic to an \'etale neighbourhood of $[(C,E)]$ in $V$. The result in \emph{loc.cit.} is formulated stack-theoretically, the same argument works for the schematic quotient. Furthermore, the schematic quotient map
$$
B\to B/\text{Aut}=B/\left(\text{Aut}/\mathbb G_m\right),
$$
is a branched covering, because $\text{Aut}/\mathbb G_m\subset \text{Aut}(C)$ is a finite group.  It is unramified outside of the locus of objects with non-trivial automorphisms. By \cite[Lemma 4.4]{NR}, any component of $N$ of codimension at least two is contained in the singular locus of $V$. By Lemma \ref{cane}, $\text{cod}(V,N)\geq 2$,  thus $N\subset S$.\vspace{0.1cm}

We now prove (ii). Let $\Us:=\U\setminus S$ be the smooth locus of $\U$. Consider the forgetful functor $p:\Us\to M_{g,n}$. By generic smoothness, there exists an open $M^0_{g,n}\subset M_{g,n}$ such that $p:p^{-1}(M^0_{g,n})\to M^0_{g,n}$ is smooth. Without loss of generality, we can assume that if $[C]\in M^0_{g,n}$, then $C$ has no non-trivial automorphisms.\\
The locus $Q$ of non-stable objects $(C,E)$, such that $[C]\in M^0_{g,n}$, is dense in the entire locus of non-stable objects in $\U$. So, it is enough to show $Q\subset S$. The fibre of $p$ over $[C]\in M^0_{g,n}$ is isomorphic to the smooth locus of the moduli space of semistable vector bundles of rank $r$ and degree $d$ over $C$. It is known that it coincides with the locus of stable bundles, so $Q\cap\Us\empty$ must be empty, i.e. $Q\subset S$.
\end{proof}

\begin{teo}\label{stack=coarse} We have a canonical isomorphism
$$
\Br'(\Vr)\cong \Br'(\Us)
$$
of cohomological Brauer groups.
\end{teo}
\begin{proof}The variety $\U$ and the stack $\Vr$ are isomorphic along the locus of regularly stable objects, which coincides with $\Us$, by Proposition \ref{smoothloc}. In particular, the restriction induces a homomorphism of cohomological Brauer groups
$$
\Br'(\Vr)\to\Br'(\Us).
$$
If we show that the complement $\Vr\setminus \Us$ has codimension at least two, by Lemma \ref{opengeq2}, we have the theorem. By Proposition \ref{smoothloc}, this stack is the union of two closed substacks in $\Vr$:
\begin{enumerate}[(i)]
\item the stack $\mt S^d_{r,g,n}$ of objects such that $\text{Aut}(E)\neq\text{Aut}(C,x_1,\ldots,x_n,E)$,
\item the stack of non-stable vector bundles.
\end{enumerate}
The first one has codimension 2 by Lemma \ref{cane} and the second one by Theorem \ref{teo1}.
\end{proof}

The next result gives us a new representative for the Brauer class of the gerbe as a Brauer class of a projective bundle over $\Us$.
\begin{prop}\label{generatorproj}For $k$ big enough, there exists a projective bundle $\mathbb P(k)$ over $\Us$ satisfying the following two properties
\begin{enumerate}[(i)]
	\item the fibre over a point $[(C,x_1,\ldots,x_n,E)]\in\Us$ is canonically isomorphic to $\mathbb P (H^0(C,E\otimes \omega_C^k))$,
	\item its Brauer class $[\mathbb P(k)]$ is equal to the Brauer class $[\nud]$ of the gerbe restricted over $\Us$.
\end{enumerate}
Moreover, if $d\geq r(2g-1)$, we can choose $k=0$. 
\end{prop}
\begin{proof} Let $\mt U^d_{r,g,n}$ be the open substack in $\Vc$ of regularly stable objects. Consider the universal curve $\pi:\mt C\to\mt U^d_{r,g,n}$ and the universal vector bundle $\mt E$ on it. Consider the coherent sheaf $\pi_*\mt E(k):=\pi_*(\mt E\otimes \omega_\pi^k)$. The automorphism group of any object in $\mt U^d_{r,g,n}$ acts on the corresponding fibre of $\pi_*\mt E(k)$. In particular, the group $\mathbb G_m$ acts on any fibre of the bundle with weight 1. For $k$ big enough, the above sheaf is locally free. We denote by $\mathbb P(k)$ the total space of the above vector bundle without the zero section. Consider the composition $\mathbb P(k)\to \mt U^d_{r,g,n}$ with the gerbe $\mt U^d_{r,g,n}\to\Us$. We claim that the resulting map 
$$
\mathbb P(k)\to \Us
$$
is the projective bundle in the assertion. Indeed, if we take a fibre over a point $(C,E)$ is isomorphic to the quotient stack $$\Big[\big(H^0(C,E(k))\setminus\{0\} \big)\big/\mathbb G_m\Big].$$
The group $\mathbb G_m$ acts freely (because we removed the zero section) with weight $1$. Then the quotient stack is isomorphic to the projective space $\mathbb P(H^0(C,E(k)))$. The same argument works for an \'etale open of $\Us$. So locally \'etale, $\mathbb P(k)$ is a projective space. For it to be a projective bundle, the transition automorphisms between two trivializations must be linear. This can be proved by using the fact that $\pi_*\mt E(k)$ is a vector bundle. Putting all together, $\mathbb P(k)$ is a projective bundle over $\Us$ satisfying $(i)$. We now prove (ii). Consider the commutative diagram
$$
\xymatrix{
&\Br'(\mathbb P(k))\\
\Br'(\Us)\ar[r]\ar[ur]&\Br'(\mt U^d_{r,g,n}).\ar[u]
}
$$
The vertical map is an isomorphism by Lemma \ref{rescompl}. The kernel of the diagonal map is generated by Brauer class of the projective bundle $\mathbb P(k)$ by \cite[p. 193]{Gab}. The kernel of the horizontal map is generated by $[\nud]$ by Theorem \ref{kern}. So, we have $(i)$.
The last assertion comes from the fact that if $d\geq r(2g-1)$ the higher cohomology groups of any stable vector bundle vanish (see \cite[Subsection 1.1]{Sc01}). Then, in this range, the pushforward $\pi_*\mt E$ of the universal vector bundle is already a vector bundle, by \cite[Remark 8.3.11.2]{FAG}.
\end{proof}
\begin{rmk}We will see in Proposition \ref{sym=proj} that when $d>2g-2$ and $r=1$, $\mathbb P(0)$ is nothing but (an open of) the universal symmetric product $\Sym$.
\end{rmk}
Let $U_r^d (C)$ be the moduli space of stable vector bundles over a fixed curve $C$. 
It is known that a universal vector bundle over $U_r^d (C)$ exists if and only if the Brauer class of $\mathbb P(\mt E)$ of the universal projective bundle is trivial in $\Br(C\times U_r^d (C))$.\\
This fact remains true in the universal setting with markings. With no markings it is false. Indeed, by \cite[Proposition 3.3.4]{Fr16}, there exists a universal vector bundle over $U^d_{r,g,0}$ if and only if $\operatorname{gcd}(r(d+1-g),r(2g-2), d+r(1-g))=1$. On the other hand, the Brauer class of $\mathbb P(\mt E)$ lives in the Brauer group of the universal curve $\mt C$ of $U_{r,g,0}^d$, which always exists because is an open in the rigidified stack. Arguing as in the proof of Proposition \ref{generatorproj}, it can be shown the Brauer class of $\mathbb P(\mt E)$ restricted to the open $U^d_{r,g,1}\subset\mt C$ is equal to $[\nu^d_{r,g,1}]$. By Theorems \ref{kern} and \ref{stack=coarse}, it is trivial if and only if $\operatorname{gcd}(r,d)=1$.\vspace{0.1cm}

We thank Indranil Biswas for pointing out the following observation.
\begin{rmk}If $n\neq 0$, there exists a universal vector bundle over the universal curve $\mt C$ of $\Us$ if and only if the Brauer class $[\mathbb{P}(\mt E)]\in\Br'(\mt C)$ of the universal projective bundle is trivial if and only if $r$ and $d$ are coprime.

If $n=0$, it can happen that the Brauer class of the universal projective bundle is trivial ($\Leftrightarrow$ $r$ and $d$ are coprime), but the corresponding vector bundle does not satisfy the universal property (it does $\Leftrightarrow$, $r(d+1-g)$, $r(2g-2)$ and $d+r(1-g)$ are coprime). 
\end{rmk}

\end{subsection}

\end{section}

\begin{section}{The Brauer group of the universal moduli stack $\Vc$.}\label{bruniv}

In this section we will compute $\Br'(\Vc)$. We will first reduce the problem to computing it in the case of $r=1$, then to computing it for the symmetric product of $d$ copies of the universal curve over $\Mg$, for $d$ large enough. Unless otherwise stated, through the subsection, we assume that $r\geq 2$, $g\geq 3$ and $n\geq 0$.

\begin{subsection}{Reduction steps}
	The aim of this subsection is to show the following
	\begin{prop}\label{br=br}The cohomological Brauer group of $\Vc$ injects into the cohomological Brauer group of $\Jc$.
	\end{prop} 
	The strategy of the proof is to define an intermediary stack $\mt Ext$, sitting in the following diagram
	\begin{equation}\label{ext}
	\xymatrix{
		\mt Ext \ar[d]^j\ar[r]_{v} & \mt Vec\\
		\mt Jac&
	}
	\end{equation}
	such that 
	\begin{enumerate}[(i)]
		\item the pull-back map $\Br'(\mt Vec)\xrightarrow{v^*}\Br'(\mt Ext)$ is injective,
		\item the map $j$ makes $\mt Ext$ a vector bundle over $\mt Jac$,
	\end{enumerate}
	By (ii) the map $j^*:\Br'(\mt Jac)\to\Br'(\mt Ext)$ is an isomorphism and $0_j^*$ is the inverse. Using (i) we immediately get the assertion. Up to twisting by a power of the canonical bundle, it will be enough to show these facts for $d$ big enough.
	
	Let's start with the definition of the auxiliary stack.
	\begin{defin}Let $\Ex$ be the category fibered over $(Sch/\mathbb C)$, which associates to any scheme $S$ the groupoid of objects $$(C\to S,\underline{\sigma}, 0\to \oo_C^{r-1}\to E\to L\to 0).$$ Where the former is a family of smooth $n$-marked curves of genus $g$ over $S$ and the latter is a short exact sequence of $S$-flat sheaves over $C$, such that $E$ is a vector bundle of relative degree $d$ and relative rank $r$.
		
		A morhism between two objects over $S$
		$$
		(\varphi, f,g):(C\to S,\underline{\sigma}, 0\to \oo_C^{r-1}\to E\to L\to 0)\longrightarrow(C'\to S,\underline{\sigma}', 0\to \oo_{C'}^{r-1}\to E'\to L'\to 0)
		$$
		is a triple $(\varphi, f,g)$, where $\varphi:(C,\underline{\sigma})\to (C',\underline{\sigma}')$ is an isomorphism of families of marked curves and $f:\varphi^*L'\cong L$, $g:\varphi^*E'\cong E$ are isomorphism of sheaves over $C'$, such that the diagram
		$$
		\xymatrix{
			0\ar[r]& \oo_{C}^{r-1}\ar[r]\ar[d]^{Id}& \varphi^*E'\ar[r]\ar[d]^f& \varphi^*L'\ar[r]\ar[d]^g& 0\\
			0\ar[r]& \oo_{C}^{r-1}\ar[r]& E\ar[r]& L\ar[r]& 0
		}
		$$
		commutes.
	\end{defin}
	
	The map $v$ in the diagram (\ref{ext}) maps an object $$\alpha = (C\to S, \sigma_1,\ldots,\sigma_n, 0\to \oo_C^{r-1}\to E\to L\to 0)\in\Ex(S)$$ to $(C\to S, \sigma_1,\ldots,\sigma_n,  E)\in\Vc(S)$.
	
	Similarly, $j$ maps $\alpha$ to $(C\to S, \sigma_1,\ldots,\sigma_n,  L)\in\Jc(S)$. If $d\gg 0$, by \cite[Lemma 7.1]{DN}, the map $v$ is dominant.\\
	
	The next lemma proves point (i).
	
	\begin{lem}\label{sec}Assume $d\gg 0$. The map $v:\Ex\to\Vc$ induces an inclusion of cohomological Brauer groups.
	\end{lem}
	
	\begin{proof}We will construct a vector bundle $\mt M$ over an open substack $\mt U \subset \Vc$, with a (non-empty) open representable substack $V \subset \mt M$ sitting in the following commutative diagram 
		\begin{equation}\label{equivapp}
		\xymatrix{
			V\ar[d]^{j}\ar[r]^{\sigma} & \Ex \ar[d]^{v}\\
			\mt M\ar[r]^{p}&\Vc,
		}
		\end{equation}	
		
		As pullbacks through vector bundles and open immersions are injective, both maps $p^*, j^*$ are injective, and thus $v^*$ must be injective as well.
		
		First, up to restricting to an open subset $\mt U \subset \Vc$ and picking a sufficiently high degree, we may assume that the higher cohomology groups of any fibre of the universal vector bundle $\mt E$ along the universal curve $\pi: \mt C\to \mt U$ vanish. By \cite[Remark 8.3.11.2]{FAG}, $\pi_*(\mt E)$ is a vector bundle and commutes with base change. We remark that the result in \emph{loc. cit.} is expressed in terms of schemes, but the same holds in our situation, since the morphism $\pi$ is representable. Moreover we can suppose that the automorphisms group of any object in $\mt U$ is equal to $\mathbb{G}_m$. Now consider the vector bundle $\mt M = \pi_*(\mt E)^{r-1}$ over $\mt U$. It represents triples 
		$$(C \rightarrow S,\, E, \mt O_C^{r-1} \rightarrow E)$$
		where $(C \rightarrow S, E)$ is an object in $\mt U$. Inside $\mt M$, consider the open subset $V$ of elements such that the map $O_C^{r-1} \rightarrow E$ is injective and its cokernel is locally free. By \cite[Lemma 7.1]{DN} the image of this subset in $\mt U$ is everywhere dense, so in particular $V$ is non-empty. Additionally note that the objects in $V$ have trivial automorphisms, so $V$ is representable, i.e. an algebraic space (one can show that it is actually a scheme). By construction the universal curve $\mt C_V$ over $V$ comes equipped with a short exact sequence 
		$$ 0 \rightarrow \mt O_{\mt C_V}^{r-1} \rightarrow \mt E_{\mt C_V} \rightarrow L \rightarrow 0$$
		where $L$ is a line bundle. This induces the desired map $\sigma$.

	\end{proof}
	
	We now prove the point (ii).
	
	\begin{lem}\label{exvec}Assume $d\gg 0$. The morphism $j:\Ex\to\Jc$ endows the stack $\Ex$ with the structure of vector bundle over $\Jc$.
	\end{lem}
	
	\begin{proof}As in the proof of \cite[Proposition A.3(ii)]{Ho10}, the fibre of $j$ over a geometric point $(C,L)$ is the quotient stack of the affine space $\text{Ext}^1(L,\oo_C^{r-1})$ modulo the trivial action of the group $\text{Hom}(L,\oo_C^{r-1})$. Since $d$ is positive, the latter group is trivial. Moreover, $\text{Ext}^1( L,\oo_C^{r-1})$ is equal to $r-1$ copies of $H^1(C, L^{-1})=H^0(C,L\otimes\omega_C)^\vee$. The latter equality comes by Serre duality. \\
		The same construction can be done in families: let $\pi:\mt C\to \Jc$ be the universal curve and $\mt L$ be the universal bundle over it. Assume that $d$ is big enough such that $H^1(C, L\otimes\omega_C)=0$ for any geometric point $(C,L)\in\Jc$. Then we have a natural isomorphism of stacks
		\begin{equation}\label{auxvec}
		\left(\pi_*(\mathcal L\otimes \omega_{\pi}\right)^\vee)^{r-1}\cong \Ex,
		\end{equation}
		endowing $\Ex$ with a vector bundle structure over $\Jc$.
	\end{proof}
	
\end{subsection}

\begin{subsection}{The universal symmetric product $\Sym$.}

Consider the $d$-th symmetric product $\Sym$ of the universal curve $\mt C_{g,n}\to\Mg$. We want to show that
$$
\Br'(\Jc)=\Br'(\Sym).
$$
Unless otherwise stated, through the subsection, we assume that $g\geq 3$ and $n\geq 0$.\\

 The $d$-th symmetric product admits the following moduli interpretation: for each scheme $S$ it associates the groupoids of pairs $(C\to S, D)$, where the former is a family of smooth curves of genus $g$ over $S$ and $D$ is an effective divisor $D\subset C$ of relative degree $d$ over $S$. Consider the universal Abel-Jacobi morphism
$$
\begin{array}{cccc}
A_d:&\Sym&\longrightarrow& \Jc\\
& (C\to S, D)&\mapsto &(C\to S, \oo_C(D)).
\end{array}
$$
It is well known that when $d>2g-2$, the map is surjective. We want to show that the Abel-Jacobi map allows us to describe $\Sym$ as an open subset of a vector bundle over $\Jc$.

\begin{prop}\label{sym=proj}Assume $d> 2g-2$. We have an isomorphism of stacks
$$
\Sym\cong \pi_*(\mt L)\setminus \{ 0_{\Jc}\},
$$
where $\mt L$ is the universal line bundle over the universal curve $\pi:\mt C\to \Jc$ and $0_{\Jc}$ is the zero section of the vector bundle $\pi_*(\mt L)$. In particular, $\Sym$ is an open subset of a vector bundle over $\Jc$, whose complement is of codimension greater than $1$.
\end{prop}

\begin{proof}We set $\mt T:=\pi_*(\mt L) \setminus \{ 0_{\Jc}\}$ and remove the markings from the notation. Let $(C,L)$ be a point in $\Jc$, by Serre duality $H^1(C,L)\cong H^0(C,L^{-1}\otimes\omega_C)^*$ and the latter is zero by assumptions on $d$. Then, as in the proof of Lemma \ref{sec}, the stack $\mt T$ represents triples 
$$(C \rightarrow S,\, E, \mt O_C \rightarrow L),$$
where $(C \rightarrow S, E)$ is an object in $\Jc$ and $\mt O_C\to L$ is injective (because we removed the zero section). We want to define a morphism of stacks 
$$A:\Sym\longrightarrow \mt T.$$
Let $(C\to T,D)$ be an object in $\Sym$. Consider the exact sequence
\begin{equation}
0\to \oo_C(-D)\xrightarrow{s} \mt O_C\to \mt O_D\to 0
\end{equation}
of $T$-flat sheaves over $C$. We remark that $\oo_C(-D)$ is a line bundle, so we can consider the morphism
$
s^{\vee}:\mt O_C\to \oo_C(D)
$
 dual to $s$. The functor $A$ is defined as it follows: we send a pair $(C\to S, D)$ to $(C\to S, \mt O_C\xrightarrow{s^{\vee}} \oo_C(D))$. Conversely, any object $(C\to S,\mt O_C\xrightarrow{m} L )$ in $\mt T$ defines an exact sequence
\begin{equation}
0\to L^{-1}\xrightarrow{m^{\vee}} \mt O_C\to M\to 0,
\end{equation}
of $S$-flat sheaves over $C$. Let $D(m)$ be the closed subscheme attached to it, it is  a relative effective divisor on $C\to S$ of degree $d$. So, we have defined a morphism $$B:\mt T\to \Sym,$$ which maps $(C\to S,\mt O_C\xrightarrow{m} L )$ to $(C\to S,D(m) )$. Checking that $B$ is the inverse of $A$ is a straightforward computation.
\end{proof}

By Lemma \ref{vec}, \ref{opengeq2}, we have immediately
\begin{cor}\label{sym=jac}Assume $d> 2g-2$. The universal Abel-Jacobi map induces an isomorphism
$$
\Br'(\Sym)\cong\Br'(\Jc)
$$
of cohomological Brauer groups.
\end{cor}

\end{subsection}

\begin{subsection}{The Brauer group of the universal symmetric product $\Sym$.}

Unless otherwise stated, through the subsection, we assume that $g\geq 3$ and $n\geq 0$. The aim of this section is proving the following

\begin{teo}\label{brsym}Assume $d\geq 2$. We have an inclusion of cohomological Brauer groups
$$
\Br'(\Sym)\subset \Br'(\mathcal M_{g,d+n}).
$$
\end{teo}
Before continuing, we fix some notations. Let $T$ be an irreducible, noetherian and regular Deligne-Mumford stack and $C\to T$ be a flat family of smooth curves over $S$. We denote by $$C^{d}:=\underbrace{C\times_T\ldots\times_T C}_{d \text{ times}}$$ the relative $d$-th fibred product. It comes equipped with a canonical map onto $T$, whose geometric fibres over a point $t\in T$ are the usual $d$-th product of the curve $C_t$.
	The symmetric group $S_d$ acts fiberwise on $C^d\to T$ by permuting the factors. There is a schematic quotient 
	$$
	Sym^dC:=C^d/S_d
	$$
equipped with a representable map $Sym^dC \rightarrow T$. It is an irreducible smooth Deligne-Mumford stack. We will also consider the stacky quotient $[C^d/S_d]$ which is again an irreducible smooth Deligne-Mumford stack, but the map $[C^d/S_d] \rightarrow T$ is not representable.

Observe that $\Sym= \mt C^d_{g,n}/S_d$, where $\mt C^d_{g,n}$ is the relative $d$-th fibred product of the universal curve $\mt C_{g,n}\to \mt M_{g,n}$ (see also Subsection \ref{app}).

With the first lemma, we give a description, under certain hypotheses, of the second cohomology group of the quotient stack.

\begin{lem}\label{lemma1}Let $C\to T$ be as above. If $n$ is an integer such that $H^1(C^d,\mu_n)=0$, then the sequence
\begin{equation}\label{brauerstack}
0\to H^2(\mt BS_d, \mu_{n})\to H^2([C^d/S_d],\mu_{n})\to H^2(C^d,\mu_{n})^{S_d}
\end{equation}
is exact.
\end{lem}

\begin{proof}We set $\mathcal F:= \mu_{n}$. Consider the natural map $\pi:[C^d/S_d]\to \mt BS_d$ and the spectral sequence 
$$
H^p(\mt BS_d,R^q\pi_*\mathcal F)\to H^{p+q}([C^d/S_d],\mathcal F)
$$	
attached to it. This induces a 7-term exact sequence in low-degrees
\begin{alignat*}{2}
0 &\to H^1(\mt BS_d,\mathcal F)\to H^1([C^d/S_d],\mathcal F)\to H^0(\mt BS_d,R^1\pi_*\mathcal F)\to H^2(BS_d,\mathcal F)\to\\
& \to \text{ker}\left\{H^2([C^d/S_d],\mathcal F)\to H^0(\mt BS_d,R^2\pi_*\mathcal F)\right\}\to H^1(\mt BS_d,R^1\pi_*\mathcal F).
\end{alignat*}
Observe that $H^i(\mt BS_d,R^1\pi_*\mt F)=H^i(S_d,H^1(C^d,\mt F))=0$, by hypothesis. Then the lemma follows from the fact that $H^0(\mt BS_d,R^2\pi_*\mathcal F)=H^2(C^d,\mathcal F)^{S_d}.$
\end{proof}
The next lemma describes the second cohomology group of the schematic quotient.

\begin{lem}\label{lemma2}Let $C\to T$ be as above. If $n$ is an integer such that $H^1(C^d,\mu_n)=0$, then we have an inclusion of abelian groups
	$$
	H^2(Sym^dC,\mu_{n})\subset H^2(C^d,\mu_{n})^{S_d}.
	$$
	\end{lem}

\begin{proof}Using the same sequence in the proof of Lemma \ref{lemma1} for the sheaf $ \mathbb G_m$, we get a short exact sequence
\begin{equation}\label{ex}
	0\to Hom(S_d,\mathbb G_m(T))\to\Pic([C^d/S_d])\to \Pic(C^d)^{S_n}=\Pic(Sym^dC)\to 0 .
\end{equation}
The right exactness follows because the map $\pi:[C^d/S_d]\to \mt BS_d$ has a section. Furthermore, the sequence splits canonically: the section $\Pic(Sym^dC)\to \Pic([C^d/S_d])$ is given by identifying $\Pic(Sym^dC)$ with the subgroup of line bundles of $[C^d/S_d]$, where the stabilisers act trivially. Using the Kummer sequence, we obtain a morphism of short exact sequences:
	$$
	\xymatrix{
		0\ar[r]&\frac{\Pic(Sym^dC)}{n\Pic(Sym^dC)}\ar[r]\ar[d]&H^2(Sym^dC,\mu_n)\ar[r]\ar[d] &\Br(Sym^dC)_n\ar[d]\ar[r]& 0\\
		0\ar[r] &\frac{\Pic([C^d/S_d])}{n\Pic([C^d/S_d])}\ar[r] &H^2([C^d/S_d],\mu_n)\ar[r] &\Br([C^d/S_d])_n\ar[r]& 0
	}
	$$
The third vertical arrow is injective because the stacky quotient and the schematic one are birationals. Observe that the first vertical arrow is a section of the exact sequence (\ref{ex}), tensored by $\mathbb Z/n\mathbb Z$, which remains injective because the sequence splits. By the snake lemma the arrow in the middle is also injective. By Lemma \ref{lemma1}, to prove the lemma it is enough to show that the image of the inclusion
$$
i:H^2(\mt BS_d,\mu_n)\to H^2([C^d/S_d],\mu_n)
$$
does not meet $H^2(Sym^dC,\mu_n)$. We will prove it by contradiction. Assume there exists $0\neq \alpha\in H^2(\mt BS_d,\mu_n)$ such that $i(\alpha)\in H^2(Sym^dC,\mu_n)$. Before we continue, we need the following observation, whose proofs are straightforward.
\begin{enumerate}[(i)]
	\item If $X$ is any geometric fibre of the morphism $C\to T$, then 
	$$
	H^2(\mt BS_d,\mu_n)\subset \text{ker}\left\{H^2([X^d/S_d],\mu_n)\to H^2(X^d,\mu_n)^{S_d}\right\}.
	$$
	\item The map above $H^2(\mt BS_d,\mu_n)\subset H^2([X^d/S_d],\mu_n)$ factors through $H^2([C^d/S_d],\mu_n)$.
\end{enumerate}
Consider the following commutative diagram
$$
\xymatrix{
H^2([C^d/S_d],\mu_n)\ar[r]&H^2([X^d/S_d],\mu_n)\\
H^2(Sym^dC,\mu_n)\ar@{^{(}->}[u]\ar[r]&H^2(Sym^dX,\mu_n)\ar[u]
}
$$
Since the first vertical map is an inclusion, point $(ii)$ implies $i(\alpha)|_{Sym^d X}\neq 0$. By point $(i)$, it follows that the map 
\begin{equation}\label{et}
H^2(Sym^dX,\mu_n)\to H^2(X^d,\mu_n)
\end{equation}
sends $i(\alpha)|_{Sym^d X}$ to zero. In particular, after choosing of an isomorphism $\mu_n\cong\mathbb Z/n\mathbb Z$, by the comparison theorem and the analogous map in the analytic setting
\begin{equation}\label{an}
H^2_{\mbox{an}}((Sym^dX)^{an},\mathbb Z/n\mathbb Z)\to H^2_{\mbox{an}}((X^d)^{an},\mathbb Z/n\mathbb Z)
\end{equation}
is not injective. We will show that this map must be injective, obtaining a contradiction and concluding the proof. Consider the pull-back map of cohomologies with coefficients in a $\mathbb Z$-module $M$:
\begin{equation}\label{MB}
H^2_{\mbox{an}}((Sym^dX)^{an},M)\to H^2_{\mbox{an}}((X^d)^{an},M).
\end{equation}
In \cite{Mac}, it has been proved that this map is an injective homomorphism of torsion-free groups, when $M=\mathbb Z$. Using the explicit description in \cite[Theorem 3.1]{BDH}, we can see that also the cokernel is free. Indeed, following the notation of \emph{loc. cit.}, a free basis for $H^2_{\mbox{an}}((X^d)^{an},\mathbb Z)$, resp. $H^2_{\mbox{an}}((Sym^dX)^{an},\mathbb Z)$, is given by
$$
\begin{cases}
 \lambda_i^j\cup\lambda_{i'}^{j'}, &\text{ for  }i\neq i' \text{ and }1\leq j<j'\leq d\\
 \eta^j, &\text{ for }1\leq j\leq d
\end{cases},\text{ resp. }
\begin{cases}
\sum_{j<j'}\lambda_i^j\cup\lambda_{i'}^{j'}-\lambda_{i'}^j\cup\lambda_{i}^{j'}, &\text{ for  }i\neq i',\\
\sum_j\eta^j.
\end{cases}
$$
By direct computation, the cokernel of (\ref{MB}) for $M=\mathbb Z$ is free with basis induced by
$$
\begin{cases}
\lambda_i^j\cup\lambda_{i'}^{j'}, &\text{ for  }i\neq i' \text{ and }2\leq j<j'\leq d,\\
\eta^j, &\text{ for }2\leq j\leq d.
\end{cases}
$$
In particular, if we apply $(-)^*:=\text{Hom}_{\mathbb Z}(-,\mathbb Z)$ to the map (\ref{MB}) when $M=\mathbb Z$, it becomes surjective. Then, by applying $\text{Hom}_{\mathbb Z}(-,\mathbb Z/n\mathbb Z)$, we obtain the injective homomorphism
\begin{equation}\label{dualMB}
\text{Hom}_{\mathbb Z}\big((H^2_{\mbox{an}}(Sym^dX)^{an},\mathbb Z)^*,\mathbb Z/n\mathbb Z\big)\to \text{Hom}_{\mathbb Z}\big(H^2_{\mbox{an}}((X^d)^{an},\mathbb Z)^*,\mathbb Z/n\mathbb Z\big).
\end{equation}
We will show that (\ref{dualMB}) is the same as the morphism (\ref{an}).
By the universal coefficient theorem, there exists an exact sequence
\begin{equation}\label{ufc}
0\to \text{Ext}^1_{\mathbb Z}(H_{i-1}(V,\mathbb Z),M)\to H^i(V,M)\to \text{Hom}_{\mathbb Z}(H_i(V,\mathbb Z),M)\to 0,
\end{equation}
where $M$ is a $\mathbb Z$-module and $V$ is either $(Sym^dX)^{an}$ or $(X^d)^{an}$. In particular, when $M=\mathbb Z$, the torsion subgroup of $H^i(V,\mathbb Z)$ is equal to the torsion subgroup of $H_{i-1}(V,\mathbb Z)$. Since all the groups $H^i(V,\mathbb Z)$ are free, all the homology groups $H_i(V,\mathbb Z)$ are free as well. Furthermore, the sequence (\ref{ufc}) gives
\begin{equation}\label{dual}
H^i(V,\mathbb Z)\cong H_i(V,\mathbb Z)^*,\text{ or equivalently } H^i(V,\mathbb Z)^*\cong H_i(V,\mathbb Z).
\end{equation}
In particular, when $M=\mathbb Z/n\mathbb Z$, the first term of the sequence (\ref{ufc}) vanishes and we have
\begin{equation}\label{hom}
H^i(V,\mathbb Z/n\mathbb Z)\cong \text{Hom}_{\mathbb Z}(H_i(V,\mathbb Z),\mathbb Z/n\mathbb Z).
\end{equation}
We remark that these isomorphisms are all functorial in $V$. Putting together (\ref{dual}) and (\ref{hom}) for $i=2$, we obtain the identification between (\ref{dualMB}) and (\ref{an}).
Then the map in (\ref{an}) not being injective is a contradiction, which shows that $\alpha$ must be $0$. 
\end{proof}

Now we can finally complete the proof of our theorem. We will use some results on $\Mg$ and on the $d$-th product of the universal curve $\Cid \rightarrow \Mg$ which are proven in the next section. \vspace{0.1cm}\\
\begin{dimo} \emph{Theorem \ref{brsym}}. Let $p: \Ci\to\Mg$ be the universal curve over the moduli space of $n$-marked curves. We have a natural morphism of exact sequences:
$$
\xymatrix{
0\ar[r]& \frac{\Pic(\Sym)}{k\Pic(\Sym)}\ar[r]\ar[d]& H^2(\Sym,\mu_k)\ar[r]\ar[d]&\Br'(\Sym)_k\ar[r]\ar[d]&0\\
0\ar[r]& \left(\frac{\Pic(\Cid)}{k\Pic(\Cid)}\right)^{S_d}\ar[r]& H^2(\Cid,\mu_k)^{S_d}\ar[r]&\left(\Br'(\Cid)_k\right)^{S_d}
}
$$
Observe first that, the $d$th-product $\Cid$ contains the moduli stack $\mathcal M_{g,n+d}$ (see Section \ref{app}). In particular, $\Br'(\Cid)\subset \Br'(\mathcal M_{g,n+d})$. We remark that $\Pic(\Sym)=\Pic(\Cid)^{S_d}$. This implies that the cokernel of the first vertical map is a subgroup of $H^1(S_d,\Pic(\Cid))$, which is zero by Proposition \ref{van}. So the first vertical map is surjective. 

By the snake lemma, to prove the Theorem it is enough to show that the arrow in the middle is injective. By Lemma \ref{lemma2}, we just need to check that the universal $d$-th product satisfies 
\begin{equation}\label{gr}
H^1(\mt C^d_{g,n},\mu_k)=0,
\end{equation}
for any $k$. The group above sits in the middle of the following exact sequence
\begin{equation}
0\to H^1(\Mg,\mu_k)\to 	H^1(\mt C^d_{g,n},\mu_k)\xrightarrow{f} H^0(\Mg,R^1q_*\mu_k),
\end{equation}
where $q:\Cid\to \Mg$ is the canonical morphism. So, it is enough to check the vanishing of the term on the left and of the image of $f$. The left-hand side group vanishes by Theorem \ref{cohmg}(i), when $g\geq 4$ and Remark \ref{cohmgrmk}, when $g=3$. Note that the $\text{Im}f$ in $H^0(\Mg,R^1p_*\mu_k)$ is equal to the $k$-torsion of relative Picard group $\Pic(\Cid)/\Pic(\Mg)$. This is zero due to the description of the the Picard group of $\mt C^d_{g,n}$ in Theorem \ref{pic}. 
\end{dimo}

\begin{rmk}\label{topol}
	We would like to remark a possible different route to proving (most of) the theorem, which the authors were following before coming up with the present argument. The objective is to prove that $\Br'(\Jc)=0$.
	
	In \cite{ErWill} J. Erbert and O. Randall-Williams use topological methods to compute the homology and cohomology of a class of homotopy spectra they denote as $\mathcal S^{n}_{g}(k)$, which classify families of surfaces of genus $g$ equipped with $n$ sections and a line bundle of degree $k$. They also show that when $n=0$ the spectrum $\mathcal S_{g}(k)$ is homotopically equivalent to the topological stack $\mathrm{Hol}^g_k$ which classifies families of Riemann surfaces of genus $g$ together with a line bundle of degree $k$. In section \cite[4.5]{ErWill} they state, without proof, that the analytic stack $\mathrm{Hol}^g_k$ should be isomorphic to the analyfication of $\mt J ac^k_g$.
	
	We have not proven it, but we believe a stronger version of the statements above to be true, i.e. that $\mathcal S^{n}_{g}(k)$ is homotopically equivalent to the topological stack $\mathrm{Hol}^g_{k,n}$, which classifies families of Riemann surfaces of genus $g$ together with $n$ sections and a line bundle of degree $k$, and that $\mathrm{Hol}^g_{k,n}$ is isomorphic to the analyfication of $\mt J ac^k_{g,n}$.
	
	Assuming this, we could proceed as follows. First, we pick an equivariant approximation $X$ of $\mt J ac^k_{g,n}$ in such a way that all of the invariants we are interested in (i.e. Picard and Brauer groups, étale and singular cohomology up to a sufficient degree) are maintained. Now on $X$, we use the exact sequence 
	$$ 0 \rightarrow H^2_{\text{sing}}(X,\mathbb{Z})/\Pic(X)\otimes \mathbb{Q}/\mathbb{Z} \rightarrow \Br'(X) \rightarrow H^3_{\text{sing}}(X,\mathbb{Z})_{\text{tors}} \rightarrow 0$$
	from \cite{Sch05} to conclude that $\Br'(X)$ is a subgroup of the torsion of $H^3(X,\mathbb{Z})$, which is in turn isomorphic to the torsion in $H^3(S_{g,n}(k),\mathbb{Z})$. 
	
	All that is left is to show that this group is trivial, which by the universal coefficient theorem is implied by $H_2(S_{g,n}(k),\mathbb{Z})$ being free. We use \cite[Theorem C]{ErWill}, or more specifically the argument in its proof, to prove our main theorem when $n=0, g \geq 5$, and theorem 1.8 of R. Cohen and I. Madsen's paper \cite{CohMad}, together with Hurewicz's theorem to prove the theorem in the case $g \geq 4, n > 0$. Note that this approach does not produce any result for $g \in \lbrace 3,4 \rbrace, n=0$ and $g = 3, n>0$, contrary to the more algebraic one we decided to follow.
\end{rmk}
\end{subsection}
\end{section}

\begin{section}{Some results on $\Mg$ and $\Cid$.}\label{someresults}

In this section we prove some auxiliary results concerning $\Mg$ and $\Cid$ that were used in the proof of our main theorem. Unless otherwise stated, through the section, we assume that $g\geq 3$ and $n\geq 0$.

\begin{subsection}{The groups $H^1(\Mg, \mu_k)$ and $\Br'(\Mg)$.}

We first prove some facts about the moduli stack of marked smooth curves $\Mg$. The following facts are probably well-known to experts, but no reference is known to the authors.

\begin{teo}\label{cohmg}Let $g,n\in\mathbb Z$, such that $g\geq 4$ and $n\geq 0$. We have the following:
\begin{enumerate}[(i)]
	\item $H^1(\Mg,\mu_k)=0$ for any $k$.
	\item $\Br'(\Mg)=0$.
\end{enumerate}
\end{teo}

\begin{proof}
$(i)$. Since we are working over the complex numbers, there is a non-canonical isomorphism $H^1(\Mg,\mu_k)=H^1(\Mg,\mathbb Z/k\mathbb Z)$. By the comparison theorem, the latter group is isomorphic to the first cohomology group in the analytic category $H^1_{\mbox{an}}(\Mg^{an},\mathbb Z/k\mathbb Z)$, where $\Mg^{an}$ is the underlying analytic orbifold associated to $\Mg$. By the universal coefficient theorem, we have an exact sequence
\begin{equation}
0\to\text{Ext}^1_{\mathbb Z}(H_0(\Mg,\mathbb Z),\mathbb Z/k\mathbb Z)\to H^1_{\mbox{an}}(\Mg^{an},\mathbb Z/k\mathbb Z)\to \text{Hom}_{\mathbb Z}(H_1(\Mg,\mathbb Z),\mathbb Z/k\mathbb Z)\to 0.
\end{equation}
So the vanishing is due to $H_0(\Mg^{an},\mathbb Z)=\mathbb Z$ and $H_1(\Mg^{an},\mathbb Z)=0$. The latter equality is a consequence of \cite[Lemma 1.1]{Har83}, see also \cite[Appendix]{AC87}.\\
$(ii)$. Using the Kummer sequence in the analytic and algebraic category, we obtain a natural morphism of exact sequences
$$
\xymatrix{
	\Pic(\Mg) \ar[d] \ar[r]^{\cdot k} & \Pic(\Mg) \ar[d] \ar[r] & H^2(\Mg, \mu_k) \ar[r] \ar[d] & \Br'(\Mg)_k \ar[r] \ar[d] & 0\\
	\Pic(\Mg^{an}) \ar[r]^{\cdot k} & \Pic(\Mg^{an}) \ar[r] & H^2_{\mbox{an}}(\Mg^{an}, \mu_k) \ar[r] & \Br_{\mbox{an}}(\Mg^{an})_k \ar[r] & 0\\
}
$$
where $\Br_{\mbox{an}}(\Mg^{an})_k$ is the subgroup of $k$-torsion elements of $H^2_{\mbox{an}}(\Mg^{an},\oo_{\mbox{an}}^\times)$. The Arbarello-Cornalba computation of the Picard group \cite{AC87} of the moduli space of marked curves works either in the analytic category or in the algebraic one. In particular, it implies that the first two vertical maps are isomorphisms. The third one is an isomorphism because of the comparison theorem. By the snake lemma, we have
$$
\Br'(\Mg)=\Br_{\mbox{an}}(\Mg^{an}).
$$
So, we need to show that the right-hand side vanishes. By \cite[Proposition 1.1]{Sch05}, we have an exact sequence of groups
$$
0\to \frac{H^2(\Mg^{an},\mathbb Z)}{\Pic(\Mg^{an})}\otimes\mathbb Q/\mathbb Z\to \Br_{\mbox{an}}(\Mg^{an})\to H^3_{\mbox{an}}(\Mg^{an},\mathbb Z)_{\text{tors}}\to 0.
$$
The vanishing of the first group is a consequence of \cite{AC87}. By the universal coefficient theorem,
$$
H^3_{an}(\Mg^{an},\mathbb Z)_{\text{tors}}=H_2(\Mg,\mathbb Z)_{\text{tors}},
$$
and the latter group is $0$ by \cite[Theorem 1]{KS}. This concludes the proof.
\end{proof}

\begin{rmk}\label{cohmgrmk} The proof of point $(i)$ works without any change when $g=3$. Repeating the same proof of point $(ii)$, when $g=3$, we can prove that $\Br'(\mt M_{3,n})$ is contained in the the torsion part of $H^3_{an}(\mt M^{an}_{3,n},\mathbb Z)$, which is unknown at the moment. For partial results in this direction, see \cite{KS}.
\end{rmk}

\end{subsection}

\begin{subsection}{The $S_d$-action on $\Pic(\Cid)$.}\label{app}
In this subsection, we prove some facts about the relative $d$-th product:
$$\Ci^{d}:=\underbrace{\Ci\times_{\Mg}\ldots\times_{\Mg} \Ci}_{d \text{ times}}$$
of the universal curve $\Ci\to\Mg$ of the moduli stack $\Mg$.

 The product $\Cid$ can be equivalently thought of as the moduli stack of objects $(C,x_1,\ldots,x_{n+d})$, where $C$ is a smooth curve of genus $g$ and $x_1,\ldots,x_{n+d}$ are points of $C$ such that $x_i\neq x_j$ for any $i$ and $j$ bigger than $d$. We call such stack \emph{Universal $d$-th Product of the Universal Curve over $\Mg$}.
	
	It is an irreducible noetherian smooth Deligne-Mumford stack. It contains the moduli stack $\mathcal M_{g,n+d}$ as the open substack of pointed curves where all the points are distinct. The complement has codimension one and the irreducible components are the divisors
	$$
	\mt D_{i,j}:=\left\{(C,x_1,\ldots,x_{n+d})\in\Cid|x_i=x_j\right\},
	$$
	for $1\leq i\leq d$ and $i<j\leq n+d$. We denote by the same symbols the associated line bundles. Now, since $\Cid$ is a moduli stack, it carries a universal object, i.e. it has a universal curve $\pi: \mt X\to\Cid$ and $n+d$ sections, which we denote by $s_i$ for $1\leq i\leq n+d$. Using these data, we define the following line bundles $$\Psi_i:=s_i^*\omega,$$
	where $\omega$ is the relative cotangent bundle of the universal curve $\mt X\to  \Cid$. We have the following
	\begin{teo}\label{pic}Assume $d\geq 2$. The Picard group of $\Cid$ is freely generated by $\mt D_{i,j}$ for $1\leq i\leq d$ and $i<j\leq n+d$, $\Psi_k$ for $1\leq k\leq n+d$ and the Hodge line bundle $\lambda$.
	\end{teo}
	
	\begin{proof}When $n=0$, a proof can be found in \cite[pp. 843-844]{Kou91}. The case $n>0$ follows from the previous one, by identifying $\Cid$ as an open of $\mt C^{d+n}_{g,0}$ and then using the restriction map for Picard groups.
	\end{proof}
	
	The symmetric group $S_d$ acts fiberwise on $\Cid\to \Mg$ by permuting the factors. Using the moduli interpretation above, it acts by permuting the first $d$ points and fixing the others. This	induces an action on the Picard group of $\Cid$. The rest of the subsection is devote to prove the following.
	
	\begin{prop}\label{van}Assume $d\geq 2$. The cohomology group $H^1(S_d,\Pic(\Cid))$ is trivial.
	\end{prop}
	
We denote by $T_d$ the standard representation of $S_d$ on $\mathbb Z^d$. The group $S_d$ acts on the canonical basis $e_1,\ldots,e_d$ as it follows: $\tau.e_i=e_{\tau(i)}$ for any $\tau\in S_d$.

The symmetric product $Sym^2T_d$ has an $S_d$-action induced by the standard representation. The canonical basis  of $T_d$ induces a basis $\left\{e_{i,j}:=e_i\otimes e_j+e_j\otimes e_i\right\}_{1\leq i\leq j\leq d}$ for $Sym^2T_{d}$. The action on these elements is given by 
$$
\tau.e_{i,j}=\begin{cases}
e_{\tau(i),\tau(j)}, & \text{ if }\tau(i)\leq\tau(j),\\
e_{\tau(j),\tau(i)}, &\text{ if }\tau(i)>\tau(j),
\end{cases}
$$ for any $\tau\in S_d$. By the fact that $\tau(i)=\tau(j)$ if and only if $i=j$, we have a splitting of representations:
$$
Sym^2T_d\cong T_d\oplus U_d,
$$
where $T_d$ is the standard representation generated by $e_{i,i}$ for $1\leq i\leq d$ and $U_d$ is the subrepresentation generated by $e_{i,j}$ for $1\leq i<j\leq d$.\\

Looking at the action of $S_d$ on the Picard group of $\Cid$, we see that
	\begin{enumerate}[(i)]
		\item  $S_d$ acts trivially on $\lambda$,
		\item if $i\leq d$ then $\tau.\Psi_i=\Psi_{\tau(i)}$, for any $\tau\in S_d$,
		\item if $i> d$ then $\tau.\Psi_i=\Psi_{i}$, for any $\tau\in S_d$,
		\item if $j\leq d$ then $$\tau.\mt D_{i,j}=\begin{cases}
D_{\tau(i),\tau(j)}, & \text{ if }\tau(i)\leq\tau(j),\\
D_{\tau(j),\tau(i)}, &\text{ if }\tau(i)>\tau(j),
\end{cases}$$
for any $\tau\in S_d$.
		\item if $j>d$ then $\tau.\mt D_{i,j}=\mt D_{\tau(i),j}$, for any $\tau\in S_d$.
	\end{enumerate}
So, $\langle \lambda \rangle$ is the trivial representation $\mathbb Z_{Triv}$. By $(ii)$, $\langle\Psi_i|1\leq i\leq d\rangle$ is the standard representation $T_d$. By $(iii)$, for any $d<i\leq n+d$, the group $\langle\Psi_i\rangle$ is the trivial representation. By $(iv)$, $\langle \mt D_{i,j}|1\leq i<j\leq d\rangle$ is the representation $U_d$. By $(v)$, for $d<j\leq n+d$, the subgroup $\langle \mt D_{i,j}|1\leq i\leq d\rangle$ is the standard representation. Putting all together, we have an isomorphism of $S_d$-representations
	$$
	\Pic(\Cid)\cong\mathbb Z_{Triv}^{n+1}\oplus U_d\oplus T_d^{\oplus {n+1}}.
	$$
	So, for proving the proposition \ref{van}, it is enough to show the following
	
	\begin{lem}Assume $d\geq 2$. The cohomology groups
		 $$
		 (i)\,H^1(S_d,\mathbb Z_{Triv}),\quad (ii)\,H^1(S_d,T_d),\quad (iii)\,H^1(S_d,U_d)
		 $$
		are all trivial.
	\end{lem}
	
	\begin{proof}In general, given a $S_d$-representation $M$, the group $H^1(S_d,M)$ is the group of the \emph{crossed homomorphisms} modulo the subgroup of \emph{principal crossed homomorphisms}. A crossed homomorphism is defined as a map $f:S_d\to M$ such that $f(\tau\sigma)=f(\tau)+\tau f(\sigma)$ for $\tau$, $\sigma\in S_d$. A principal crossed homomorphism is a crossed homomorphism such that exists an $m\in M$ such that $f(\tau)=\tau.m-m$ for any $\tau\in S_d$.
			
	$(i)$. The symmetric group $S_d$ acts trivially on $\mathbb Z_{Triv}$, we have $H^1(S_d,\mathbb Z_{Triv})=\text{Hom}(S_d,\mathbb Z)$, which is zero.
		
		$(ii)$. Let $f$ be a crossed homomorphism. For $1\leq i\leq d-1$, let $\sigma_i$ be the permutation that switches $i$ and $i+1$ and fixes the other coordinates. It is known that these generate the group $S_d$. Since $\sigma_i^2=Id$, we see that 
		\begin{equation}\label{eq}
		f(\sigma_i)=a_i(e_i-e_{i+1}),\quad a_i\in\mathbb Z,
		\end{equation}
		where $e_1,\ldots, e_d$ are the canonical generators of $T_d$. Consider the element 
		$$
		m:=\sum_{j=1}^{d-1}\left(\sum_{k=1}^ja_k\right)e_{j+1}.
		$$
		By direct computation, we have that $p_m(\sigma_i)=f(\sigma_i)$ for any $i$. Here $p_m$ is the principal crossed homomorphism induced by $m$. Since $\sigma_i$ are generators for $S_d$, this implies that $p_m=f$ and so is zero in $H^1(S_d,T_d)$.
		
		$(iii)$. We will give a proof by induction on $d$. Case $d=2$. Then $U_2$ is the trivial representation. So the result follows from $(i)$. Now, assume that $H^1(S_{d-1}, U_{d-1})$ is zero. We want to show that the same happens to $S_d$. The group $S_{d-1}$ is contained in $S_d$ as the subgroup of those elements fixing $e_1$ in the standard representation $T_d$. In particular, $\sigma_i$ for $i\neq 1$ generate $S_{d-1}$. Consider the canonical restriction map
		$$
		\varphi: H^1(S_d,U_d)\to H^1(S_{d-1},U_{d}).
		$$
		We want to show that it is injective. Let $f\in H^1(S_d,U_d)$ such that $\varphi(f)=0$. It means that there exists $m\in U_d$ such that $f(\sigma_i)=\sigma_i.m-m$, for $i\neq 1$. Without loss of generality, we can assume that $f(\sigma_i)=0$ for $i\neq 1$. Using again the relation $\sigma_1^2=Id$, we see that
		$$
		f(\sigma_1)=\sum_{j=3}^d a_j(e_{1,j}-e_{2,j}), \quad a_3\ldots,a_d\in\mathbb Z.
		$$
		Since $\sigma_1$ commutes with $\sigma_i$ for $i\neq 2$, by direct computation, we see that the coefficients $a_j$ must be all equal to some integer $c$. So, if we take the element
		$$
		m:= -c\sum_{j=2}^de_{1,j},
		$$
		the associated principal crossed homomorphism $p_m$ satifies $p_m(\sigma_1)=f(\sigma_1)$ and $p_m(\sigma_i)=0$ if $i\neq 1$. So, also $f$ is zero in $H^1(S_d,U_d)$. Then $\varphi$ is injective. In particular, for proving the lemma, it is enough to show that $H^1(S_{d-1},U_d)$ is trivial. As $S_{d-1}$-representation, $U_d$ splits as direct sum of the standard representation $T_{d-1}$ (generated by $e_{1,2},\ldots, e_{1,d}$) and the representation $U_{d-1}$ (generated by $e_{i,j}$, for $2\leq i<j\leq d$). Then its first cohomology vanishes because of $(ii)$ and the inductive assumption.
	\end{proof}
	\end{subsection}
\end{section}

\bibliographystyle{alpha}

\begin{thebibliography}{BBGN07}

\bibitem[AC87]{AC87}
Enrico Arbarello and Maurizio Cornalba.
\newblock The {P}icard groups of the moduli spaces of curves.
\newblock {\em Topology}, 26(2):153--171, 1987.

\bibitem[ACV03]{ACV}
Dan Abramovich, Alessio Corti, and Angelo Vistoli.
\newblock Twisted bundles and admissible covers.
\newblock {\em Comm. Algebra}, 31(8):3547--3618, 2003.
\newblock Special issue in honor of Steven L. Kleiman.

\bibitem[AM]{AM}
Benjamin Antieau and Lennart Meier.
\newblock The {B}rauer group of the moduli stack of elliptic curves.
\newblock Available at {A}r{X}iv 1608.00851.

\bibitem[AW15]{AW}
Banjamin Antieau and Benjamin Williams.
\newblock The prime divisors of the period and index of a {B}rauer class.
\newblock {\em J. Pure Apl. Alg.}, 219(6):2218--2224, 2015.

\bibitem[BBGN07]{BBGN}
Vikraman Balaji, Indranil Biswas, Ofer Gabber, and Donihakkalu~S. Nagaraj.
\newblock Brauer obstruction for a universal vector bundle.
\newblock {\em C. R. Math. Acad. Sci. Paris}, 345(5):265--268, 2007.

\bibitem[BDH15]{BDH}
Indranil Biswas, Ajneet Dhillon, and Jacques Hurtubise.
\newblock Brauer groups of {Q}uot schemes.
\newblock {\em Michigan Math. J.}, 64(3):493--508, 2015.

\bibitem[CM11]{CohMad}
Ralph~L. Cohen and Ib~Madsen.
\newblock Stability for closed surfaces in a background space.
\newblock {\em Homology Homotopy Appl.}, 13(2):301--313, 2011.

\bibitem[DJ]{DJ}
Aise De~Jong.
\newblock A result of {G}abber.
\newblock available at \\
  https://www.math.columbia.edu/~dejong/papers/2-gabber.pdf.

\bibitem[DN89]{DN}
J.-M. Drezet and M.~S. Narasimhan.
\newblock Groupe de {P}icard des vari\'et\'es de modules de fibr\'es
  semi-stables sur les courbes alg\'ebriques.
\newblock {\em Invent. Math.}, 97(1):53--94, 1989.

\bibitem[ERW12]{ErWill}
Johannes Erbert and Oscar Randall-Williams.
\newblock Stable cohomology of the universal picard varieties and the extended
  mapping class group.
\newblock {\em Doc. Math.}, 17:417--450, 2012.

\bibitem[FGI{\etalchar{+}}05]{FAG}
Barbara Fantechi, Lothar G{\"o}ttsche, Luc Illusie, Steven~L. Kleiman, Nitin
  Nitsure, and Angelo Vistoli.
\newblock {\em Fundamental algebraic geometry}, volume 123 of {\em Mathematical
  Surveys and Monographs}.
\newblock American Mathematical Society, Providence, RI, 2005.
\newblock Grothendieck's FGA explained.

\bibitem[Fri16]{Fr16}
Roberto Fringuelli.
\newblock The {P}icard group of the universal moduli space of vector bundles on
  stable curves.
\newblock 2016.
\newblock to appear in Adv. in Math., available at 1601.04866.

\bibitem[Gab81]{Gab}
Ofer Gabber.
\newblock Some theorems on {A}zumaya algebras.
\newblock In {\em The {B}rauer group ({S}em., {L}es {P}lans-sur-{B}ex, 1980)},
  volume 844 of {\em Lecture Notes in Math.}, pages 129--209. Springer,
  Berlin-New York, 1981.

\bibitem[GB68]{Gir68}
J.~Giraud and S{\'e}minaire Bourbaki.
\newblock {\em Dix expos{\'e}s sur la cohomologie des sch{\'e}mas: expos{\'e}s
  de J. Giraud, A. Grothendieck, S.L. Kleiman ... [et al.].}
\newblock Advanced studies in pure mathematics. Masson, 1968.

\bibitem[Gir71]{Gi}
Jean Giraud.
\newblock {\em Cohomologie non ab\'elienne}.
\newblock Springer-Verlag, Berlin-New York, 1971.
\newblock Die Grundlehren der mathematischen Wissenschaften, Band 179.

\bibitem[Gro66a]{Gr66a}
Alexander Grothendieck.
\newblock Le groupe de {B}rauer : {I}. alg\`ebras d'{A}zumaya et
  interpr\'etation diverses.
\newblock {\em Séminaire Bourbaki}, 9:287--307, 1966.

\bibitem[Gro66b]{Gr66b}
Alexander Grothendieck.
\newblock Le groupe de {B}rauer : {II}. théories cohomologiques.
\newblock {\em Séminaire Bourbaki}, 9:287--307, 1966.

\bibitem[Har83]{Har83}
John Harer.
\newblock The second homology group of the mapping class group of an orientable
  surface.
\newblock {\em Invent. Math.}, 72(2):221--239, 1983.

\bibitem[HL10]{HL}
Daniel Huybrechts and Manfred Lehn.
\newblock {\em The geometry of moduli spaces of sheaves}.
\newblock Cambridge Mathematical Library. Cambridge University Press,
  Cambridge, second edition, 2010.

\bibitem[Hof07]{H07}
Norbert Hoffmann.
\newblock Rationality and {P}oincar\'e families for vector bundles with extra
  structure on a curve.
\newblock {\em Int. Math. Res. Not. IMRN}, (3):Art. ID rnm010, 30, 2007.

\bibitem[Hof10]{Ho10}
Norbert Hoffmann.
\newblock Moduli stacks of vector bundles on curves and the {K}ing-{S}chofield
  rationality proof.
\newblock In {\em Cohomological and geometric approaches to rationality
  problems}, volume 282 of {\em Progr. Math.}, pages 133--148. Birkh\"auser
  Boston, Inc., Boston, MA, 2010.

\bibitem[HS09]{HS}
Jochen Heinloth and Stefan Schröer.
\newblock The bigger {B}rauer group and twisted sheaves.
\newblock {\em Journal of Algebra}, 322(4):1187 -- 1195, 2009.

\bibitem[Kou91]{Kou91}
Alexis Kouvidakis.
\newblock The {P}icard group of the universal {P}icard varieties over the
  moduli space of curves.
\newblock {\em J. Differential Geom.}, 34(3):839--850, 1991.

\bibitem[KS03]{KS}
Mustafa Korkmaz and Andr\'as~I. Stipsicz.
\newblock The second homology groups of mapping class groups of oriented
  surfaces.
\newblock {\em Math. Proc. Cambridge Philos. Soc.}, 134(3):479--489, 2003.

\bibitem[Mac62]{Mac}
I.~G. Macdonald.
\newblock Symmetric products of an algebraic curve.
\newblock {\em Topology}, 1:319--343, 1962.

\bibitem[Mil80]{Mi}
James.~S. Milne.
\newblock {\em Etale Cohomology (PMS-33)}.
\newblock Princeton mathematical series. Princeton University Press, 1980.

\bibitem[MV14]{MV}
Margarida Melo and Filippo Viviani.
\newblock The {P}icard group of the compactified universal {J}acobian.
\newblock {\em Doc. Math.}, 19:457--507, 2014.

\bibitem[NR69]{NR}
M.~S. Narasimhan and S.~Ramanan.
\newblock Moduli of vector bundles on a compact {R}iemann surface.
\newblock {\em Ann. of Math. (2)}, 89:14--51, 1969.

\bibitem[Pan96]{P96}
Rahul Pandharipande.
\newblock A compactification over {$\overline {M}\sb g$} of the universal
  moduli space of slope-semistable vector bundles.
\newblock {\em J. Amer. Math. Soc.}, 9(2):425--471, 1996.

\bibitem[PTT15]{PTT15}
F.~Poma, M.~Talpo, and F.~Tonini.
\newblock Stacks of uniform cyclic covers of curves and their {P}icard groups.
\newblock {\em Algebraic Geometry}, 2(1):91--122, 2015.

\bibitem[Sch01]{Sc01}
Alexander Schmitt.
\newblock The equivalence of {H}ilbert and {M}umford stability for vector
  bundles.
\newblock {\em Asian J. Math.}, 5(1):33--42, 2001.

\bibitem[Sch05]{Sch05}
Stefan Schr\"oer.
\newblock Topological methods for complex-analytic {B}rauer groups.
\newblock {\em Topology}, 44(5):875--894, 2005.

\bibitem[Tay82]{Ta}
Joseph~L. Taylor.
\newblock A bigger {B}rauer group.
\newblock {\em Pacific J. Math.}, 103(1):163--203, 1982.

\end{thebibliography}
\newcommand{\etalchar}[1]{$^{#1}$}

\end{document}